\documentclass[twoside]{article}
\usepackage{pstricks, pstricks-add}
\usepackage{pgfplots}
\usepackage{amsmath,amsthm, amsfonts}
\usepackage{amsmath,amsthm, amsfonts, amssymb}
\usepackage{enumerate}
\usepackage{geometry}
\usepackage{secdot}
\usepackage{sectsty}
\allsectionsfont{\normalfont}
\numberwithin{equation}{section}

\usepackage{caption}
\theoremstyle{plain}
\newtheorem{thm}{Theorem}[section]

\newtheorem{cor}[thm]{Corollary}
\newtheorem{lem}[thm]{Lemma}
\theoremstyle{definition}
\newtheorem{defn}[thm]{Definition}
\newtheorem{exa}[thm]{Example}

\newtheorem*{question*}{Question}
\newtheorem{conj}[thm]{Conjecture}
\newtheorem{rem}[thm]{Remark}

\usepackage{enumerate}
\usepackage[utf8]{inputenc}
\usepackage[english]{babel}
 \usepackage{fancyhdr}
 \usepackage{pgf,tikz}
 \usepackage{pgfplots}
  \usepackage{pstricks-add}
\usepackage{multicol}
 \usepackage{enumerate}
 \usepackage{xcolor}
 \usepackage{geometry}
\usepackage{mathrsfs}
\usetikzlibrary{arrows}
\usepackage{mathtools}
\usepackage{wasysym}
\usepackage{float}
\usepackage{array}

\usepackage[linktocpage, bookmarks]{hyperref}

\usepackage{fancyhdr}
\usepackage{lipsum}

\author{First Author \and Second Author}
\title{HURWITZ ORBITS WITH A CELLULAR AUTOMATON}

\makeatletter
\newcommand\Author{NAQEEB UR REHMAN}
\let\Title\@title
\makeatother

\begin{document}
\title{\textsc{Hurwitz Orbits with a Cellular Automaton}}

\author{NAQEEB UR REHMAN}
\date{}
\maketitle



\paragraph{ABSTRACT.} Hurwitz orbits are the orbits of the braid group action on the powers of a rack. Hurwitz orbits for the action of the braid
group on three strands are used in \cite{21} and \cite{22} for the classification of Nichols algebras. This
classification is based on a combinatorial invariant called plague on the Hurwitz orbits. The method
to calculate plagues on the Hurwitz orbits is formulated in \cite{22} by using a cellular automaton on the Hurwitz
orbits and their quotients. By using this cellular automaton-based method we estimate the minimal plagues on the Hurwitz orbits.

\begin{center}
	\section{\textsc{Introduction}}
\end{center}
\paragraph{} The braid group on three strands is $\mathbb{B}_3=\left <\sigma_1, \sigma_2 \mid \sigma_1\sigma_2\sigma_1=\sigma_2\sigma_1\sigma_2\right>$ with center $Z(\mathbb{B}_3)=\left <\Delta\right>$, where $\Delta=(\sigma_1\sigma_2)^3=(\sigma_1\sigma_2\sigma_1)^2$. The quotient of $\mathbb{B}_3$ by its center is isomorphic to the modular group $\mathbf{PSL}(2,\mathbb{Z})$. The braid group $\mathbb{B}_3$ acts on third power of a finite rack $R$ by $\sigma_1(r, s, t) = (r \rhd s, r, t), \sigma_2(r, s, t) = (r, s \rhd t, s)$ for all $r, s, t \in R$. The orbit $\Sigma$ of this action of $\mathbb{B}_3$ is called Hurwitz orbit. The Hurwitz orbit quotient $\overline{\Sigma}$ is a set of equivalence classes of the Hurwitz orbit $\Sigma$ defined by the center $Z(\mathbb{B}_3)$. Since the action of $\mathbb{B}_3$ on any Hurwitz orbit is transitive, the action of $\mathbf{PSL}(2,\mathbb{Z})$ on the Hurwitz orbit quotient is transitive. Therefore any Hurwitz orbit quotient is a finite homogeneous $\mathbf{PSL}(2,\mathbb{Z})$-space. Moreover, any finite homogeneous $B_3$-spaces $\Sigma$ can be considered as a covering of a homogeneous $PSL(2,\mathbb{Z})$-space $\overline{\Sigma}$. These finite homogeneous $\mathbf{PSL}(2,\mathbb{Z})$-spaces can be presented in terms of Schreier coset graphs associated to the modular group $\mathbf{PSL}(2,\mathbb{Z})$ with respect to its finite index subgroups and the generators $x=\sigma_2^{-1}\sigma_1^{-1}Z(\mathbb{B}_3)$ and $y=\sigma_1\sigma_2\sigma_1Z(\mathbb{B}_3)$. In the interpretation as a $\mathbf{PSL}(2,\mathbb{Z})$-space, the vertices of the Schreier coset graph correspond to the points of the Hurwitz orbit quotient, and the edges are formed by an $x-$arrow pointing from any coset $gH$ (equivalently, point of the $\mathbf{PSL}(2,\mathbb{Z})$-space) to the coset $xgH$, and a $y-$arrow pointing from any coset $gH$ to the coset $ygH$.
\\[1pt] ${}$ \;  The Hurwitz orbits for the action of the braid group $\mathbb{B}_3$ on racks are studied in \cite{21} and \cite{22} for the classification of Nichols algebras. In this study combinatorial objects called plague, immunity, and weight on the Hurwitz orbit are defined (see Definitions \ref{3.1.}, \ref{3.8.}, and \ref{3.9.}). The immunity $\emph{imm}$ on a Hurwitz orbit is estimated by its weight $\omega$ in \cite{22}, by using a cellular automaton-based method on the Hurwitz orbits and their quotients with $xy$-cycles of length at most four. The assumption about the length of an $xy-$cycle in a Hurwitz orbit quotient can be replaced by a weaker assumption which is proposed as a following conjecture in \cite{22}.

\begin{conj}{\label{1.1.}}
\emph{Let $\Sigma$ be a covering with simply intersecting cycles of finite homogeneous $PSL(2,\mathbb{Z})$-space $\overline{\Sigma}$. Then $imm(\Sigma)\leq \omega(\Sigma)$ holds for all homogeneous $B_3$-spaces $\Sigma$.}
\end{conj}

\paragraph{} In this paper we prove Conjecture \ref{1.1.} for coverings with simply intersecting cycles of certain finite homogeneous $PSL(2,\mathbb{Z})$-spaces $\overline{\Sigma}$. In our proof we will use the robust subgraphs of certain pointed Schreier graphs of the Hurwitz orbit quotients. The paper is organized as follows. In Section \ref{2} we review the basic definitions
concerning racks, Hurwitz orbits,  quotients of the Hurwitz orbits and their labelled Schreier graphs and coverings. In Section \ref{3} we recall the study of cellular automaton, plague and immunity on the Hurwitz orbits. In Section \ref{4} we define robust subgraph of the pointed Schreier graph of the Hurwitz orbit quotient, and calculate the plagues and immunities on the coverings of certain robust subgraphs. Finally, in Section \ref{5} we prove our main result about the immunity of the coverings of certain pointed Schreier graphs.

\begin{center}
\section{\textsc{Hurwitz Orbits and Their Quotients}} \label{2}
\end{center}
In this section we recall the study of Hurwitz orbits and their quotients from \cite{21} and \cite{22}. We begin with the definition of rack.
\\[1pt] ${}$ \; A \emph{rack} is a pair $(R, \rhd)$, where $R$ is a non-empty set and
$\rhd:R\times R \longrightarrow R$ is a binary operation such that
\begin{description}
  \item[\;\;\; (R1)] the map $\phi_r : R \longrightarrow R$, defined by $\phi_r(s)=r\rhd s$, is bijective for all $r\in R$,
  \item[\;\;\; (R2)] $r \rhd (s \rhd t) = (r \rhd s) \rhd (r \rhd t)$ for all $r, s, t \in R$ (i.e., $\rhd$ is self-distributive).
\end{description}
A rack $R$ is called \emph{quandle} if $r \rhd r = r$ for all $r \in R$. A rack $R$ is called \emph{braided} if $R$ is a quandle, and for all $r, s \in R$, at least one of the equations
\begin{center}
$r \rhd (s \rhd r) = s, r \rhd s = s$,
\end{center}
holds. A group $G$ is a quandle with $r \rhd s = rsr^{-1}$ for all $r, s \in G$. This quandle is called the \emph{conjugation quandle}. Similarly, the union of conjugacy classes in $G$ is a quandle under the binary operation of conjugation.
\\[1pt] ${}$ \; Let $n$ be a positive integer. The braid group on $n$ strands is the following:
\begin{center}
$\mathbb{B}_n=\left <\sigma_1, \sigma_2,...,\sigma_{n-1}\right> / (\sigma_i\sigma_j=\sigma_j\sigma_i$ if $|i-j|\geq2, \sigma_i\sigma_j\sigma_i=\sigma_j\sigma_i\sigma_j$ if $|i-j|=1)$.
\end{center}
The braid group on $3$ strands is $\mathbb{B}_3=\left <\sigma_1, \sigma_2 \mid \sigma_1\sigma_2\sigma_1=\sigma_2\sigma_1\sigma_2\right>$. The center of $\mathbb{B}_3$ is $Z(\mathbb{B}_3)=\left <\Delta\right>$, where $\Delta=(\sigma_1\sigma_2)^3=(\sigma_1\sigma_2\sigma_1)^2$.
\\[1pt] ${}$ \; According to E. Brieskorn \cite{8}, A. Hurwitz in \cite{23} studied implicitly an action of $\mathbb{B}_n$ on the $n$ product of the conjugacy class $R$ of a group, which is therefore called the \emph{Hurwitz action} and is the following:
\begin{center}
$\sigma_i(r_1, r_2, ...,r_i,r_{i+1},...,r_n) = (r_1, r_2, ...,r_ir_{i+1}r_i^{-1}, r_i,...,r_n)$,
\end{center}
for all $r_1,r_2,...,r_n \in R$ and $i \in \{1,2,..., n-1\}$. Since the algebraic structure of racks is similar to conjugation in groups, the Hurwitz action can also be studied for racks. We recall the study of Hurwitz action on finite racks from \cite{21}.
\\[1pt] ${}$ \; Let $R$ be a finite rack. The braid group $\mathbb{B}_n$ acts on $R^n$ via the Hurwitz action:
\begin{center}
$\sigma_i(r_1, r_2, ...,r_i,r_{i+1},...,r_n) = (r_1, r_2, ...,r_i\rhd r_{i+1}, r_i,...,r_n)$,
\end{center}
for all $r_1,r_2,...,r_n \in R$ and $i \in \{1,2,..., n-1\}$. For example, the Hurwitz action of the braid group $\mathbb{B}_3$ on $R^3$ is given by:
\begin{center}
$\sigma_1(r, s, t) = (r \rhd s, r, t), \sigma_2(r, s, t) = (r, s \rhd t, s)$
\end{center}
for all $r, s, t \in R$. The orbit $\Sigma = \Sigma(r_1, ..., r_n):=\{\sigma(r_1, ..., r_n): \sigma \in \mathbb{B}_n\}$ of the Hurwitz action on $R^n$ is called the \emph{Hurwitz orbit}.
\\[1pt] ${}$ \; The Hurwitz orbits for the action of $\mathbb{B}_2$ on $R^2$ are studied in \cite{15}. The Hurwitz orbits for the action of $\mathbb{B}_3$ on $R^3$ are studied in \cite{21} and \cite{22}. Note that for a finite braided rack the possible sizes of a Hurwitz orbit are $1,3,6,8,9,12, 16, 24$ (see \cite{21}, Proposition 9). For example, for the conjugacy class of two cycles in the symmetric group $S_3$, there are three Hurwitz orbits of size $1$ and three Hurwitz orbits of size $8$.

\subsection{Hurwitz Orbit Quotients and their Coverings.}
\paragraph{} The Hurwitz orbits under the action of the braid group $\mathbb{B}_3$ can be studied as coverings of the Hurwitz orbit quotients. In this section we recall the definitions and results about the Hurwitz orbit quotients and their coverings from \cite{22}.
\\[1pt] ${}$ \; Let $R$ be a finite rack and $\Sigma \subseteq R^3$ a Hurwitz orbit. Define a relation on $\Sigma$ by:
\begin{center}
$(r, s, t)\sim (r^\prime, s^\prime, t^\prime)\Leftrightarrow \Delta^m(r, s, t)= (r^\prime, s^\prime, t^\prime)$
\end{center}
for some $m \in \mathbb{Z}$, and for all $(r, s, t), (r^\prime, s^\prime, t^\prime) \in \Sigma$. Then $\sim$ is an equivalence relation. A \emph{Hurwitz orbit quotient} is the set $\overline{\Sigma}$ of equivalence classes of $\Sigma$.
\\[1pt] ${}$ \; Let $x=\sigma_2^{-1}\sigma_1^{-1}Z(\mathbb{B}_3)$ and $y=\sigma_1\sigma_2\sigma_1Z(\mathbb{B}_3)$. Then, by construction (see \cite{26}, Appendix A), we have
\begin{center}
$\mathbb{B}_3/Z(\mathbb{B}_3)\simeq \left <x,y \mid x^3=y^2=1\right>\simeq \mathbf{PSL}(2,\mathbb{Z})$.
\end{center}
Since the braid group $\mathbb{B}_3$ acts transitively on $\Sigma$, the modular group $\mathbf{PSL}(2,\mathbb{Z})$ acts transitively on $\overline{\Sigma}$,  that is, $\overline{\Sigma}$ is a finite homogeneous $\mathbf{PSL}(2,\mathbb{Z})$-space. Note that $\overline{\Sigma}$ is also a $\mathbb{B}_3-$space on which $Z(\mathbb{B}_3)$ acts trivially. The finite homogeneous $\mathbb{B}_3$-spaces are studied in \cite{22} as coverings of finite homogeneous $\mathbf{PSL}(2,\mathbb{Z})$-spaces. The covering of a finite homogeneous $\mathbf{PSL}(2,\mathbb{Z})$-space $\overline{\Sigma}$ is defined in \cite{22} as follows.
\begin{defn}
A covering of $\overline{\Sigma}$ is a triple $(\pi, \Sigma, \overline{\Sigma})$, where $\pi: \Sigma \rightarrow \overline{\Sigma}$ is a surjective $\mathbb{B}_3$-equivariant map such that $\pi(r, s, t) = \pi(r^\prime, s^\prime, t^\prime)$ implies that $(r, s, t)= \Delta^m(r^\prime, s^\prime, t^\prime)$ for some $m \in \mathbb{Z}$ and for all $(r, s, t), (r^\prime, s^\prime, t^\prime) \in \Sigma$.
\end{defn}
\paragraph{} Note that a covering $(\pi, \Sigma, \overline{\Sigma})$ is finite if $\Sigma$ (and hence $\overline{\Sigma}$) is finite. Also, a covering $(\pi, \Sigma, \overline{\Sigma})$ is trivial if $\pi: \Sigma \rightarrow \overline{\Sigma}$ is bijective. For a covering $(\pi, \Sigma, \overline{\Sigma})$, the \emph{fiber} of an element $v\in \overline{\Sigma}$ is the subset $\pi^{-1}(v)\subseteq \Sigma$. Following the notation of \cite{22}, we write $v[*]$ for the complete fiber $\pi^{-1}(v)$ over an element $v\in \overline{\Sigma}$. Since the braid group $\mathbb{B}_3$ acts transitively on $\Sigma$ and $Z(\mathbb{B}_3)=\left <\Delta\right>$ is a normal subgroup of $\mathbb{B}_3$, we have $|\pi^{-1}(v)|=|\pi^{-1}(w)|$ for all points $v, w$ of $\overline{\Sigma}$. For a covering $(\pi, \Sigma, \overline{\Sigma})$, we write the size of any fiber by $N$, that is, $|\pi^{-1}(v)|=|\pi^{-1}(w)|=N$ for all $v, w$ of $\overline{\Sigma}$.
\\[1pt] ${}$ \; Now we recall the definitions of cycles in $\Sigma$ and $\overline{\Sigma}$, and the covering $(\pi, \Sigma, \overline{\Sigma})$ with simply intersecting cycles from \cite{22}.
\begin{defn}{\label{2.4.5.}}
 For $i \in \{1, 2\}$, a \emph{$\sigma_i-$cycle} of a homogeneous $\mathbb{B}_3-$space $\Sigma$ is a minimal non-empty subset $c_i\subseteq \Sigma$ which is closed under the action of $\sigma_i$. A covering $(\pi, \Sigma, \overline{\Sigma})$ is said to be with \emph{simply intersecting cycles} if any given $\sigma_1$-cycle $c_1$ and $\sigma_2$-cycle $c_2$ in $\Sigma$ intersect at most once, i.e., $|c_1\cap c_2|\leq 1$. An \emph{$xy-$cycle} in a homogeneous $\mathbf{PSL}(2,\mathbb{Z})$-space $\overline{\Sigma}$ is a minimal non-empty subset $C_{xy}\subseteq \overline{\Sigma}$ such that $xy.v \in C$ and $(xy)^{-1}.v \in C$ for all $v \in C$.
\end{defn}
\paragraph{} Similarly one can define $yx-$cycles $C_{yx}$. An $xy-$cycle containing a fixed element $v$ of $\overline{\Sigma}$ is written as $C_{xy}(v)$. Note that for any covering $(\pi, \Sigma, \overline{\Sigma})$, the image of a $\sigma_1$-cycle in $\Sigma$ is an $xy-$cycle in $\overline{\Sigma}$, and the image of a $\sigma_2$-cycle in $\Sigma$ is a $yx-$cycle in $\overline{\Sigma}$.

\subsection{Pointed Schreier Graphs of the Hurwitz Orbit Quotients and their Coverings.}
\paragraph{} The finite homogeneous $\mathbf{PSL}(2,\mathbb{Z})$-spaces (in particular, the Hurwitz orbit quotients $\overline{\Sigma}$) can be presented in terms of Schreier graphs associated to the modular group $\mathbf{PSL}(2,\mathbb{Z})$ with respect to the generators $x, y$ and the finite index subgroups of $\mathbf{PSL}(2,\mathbb{Z})$.  Recall that, finite homogeneous $\mathbf{PSL}(2,\mathbb{Z})$-spaces up to isomorphism are known to be in bijection with conjugacy classes of finite index subgroups of the modular group $\mathbf{PSL}(2,\mathbb{Z})$, which have been studied intensively (for example in \cite{29}).
\\[1pt] ${}$ \; The Schreier graph associated to the modular group $\mathbf{PSL}(2,\mathbb{Z})$, with respect to a finite index subgroup $H$ of $\mathbf{PSL}(2,\mathbb{Z})$ and the generators $x$ and $y$ of $\mathbf{PSL}(2,\mathbb{Z})$, is an oriented labelled graph whose vertices are the left $H-$cosets and edges are of the form $(gH, xgH)$ and $(gH, ygH)$. In the Schreier graph for $\mathbf{PSL}(2,\mathbb{Z})$, an $x-$arrow points from any coset $gH$ to the coset $xgH$ and a $y-$edge points from any coset $gH$ to the coset $ygH$. Since, $\mathbf{PSL}(2,\mathbb{Z})=\left <x,y \mid x^3=y^2=1\right>$, the Schreier graph associated to $\mathbf{PSL}(2,\mathbb{Z})$ consists of oriented triangles of $x-$arrows (slid arrow) and double $y-$edges (dashed lines). Usually, instead of a double $y-$edge, a single edge or dashed line is displayed in the Schreier graph for $\mathbf{PSL}(2,\mathbb{Z})$. The fixed points of $x$ are shown by solid loop or circle with an arrow on them and the fixed points of $y$ are shown by dashed loop or circle.
\\[1pt] ${}$ \; Let $\mathcal{G}=(V,E)$ be the Schreier graph of a finite homogenous $\mathbf{PSL}(2,\mathbb{Z})$-space $\overline{\Sigma}$ of size $n$ with vertex set $V=V(\mathcal{G})$ and edge set $E$. We call $\mathcal{G}$ as a \emph{pointed Schreier graph} of $\overline{\Sigma}$ if $\mathcal{G}$ has a distinguished vertex, say, $v_0$. Note that $\mathcal{G}$ consists of $x-$triangles, $x-$loops, $y-$edges and $y-$loops. We will write
 \begin{align*}
   V_x:=\{v\in V(\mathcal{G})|x(v)=v\},
   V_y:=\{v\in V(\mathcal{G})|y(v)=v\},
   V_{xy}:=\{v\in V(\mathcal{G})|xy(v)=v\}.
 \end{align*}
If needed, we will denote $\mathcal{G}$ of size $n$ explicitly by $\mathcal{G}_{n\{l_1, l_2,...,l_k\}}$, where $\{l_1, l_2,..., l_k\}$ is a multiset of the lengths of $xy$-cycles (or $yx$-cycles) of $\mathcal{G}$ for non-negative integer $k$.
\\[1pt] ${}$ \; Note that in the interpretation as a homogeneous $\mathbf{PSL}(2,\mathbb{Z})$-space (in particular, the Hurwitz orbit quotient $\overline{\Sigma}$), the vertices of the Schreier graph for $\mathbf{PSL}(2,\mathbb{Z})$ correspond to the points of the $\mathbf{PSL}(2,\mathbb{Z})$-space. The Schreier graph of a $\mathbf{PSL}(2,\mathbb{Z})$-space $\overline{\Sigma}$ can also be used to display the covering $(\pi, \Sigma, \overline{\Sigma})$. The \emph{graph of the covering} $(\pi, \Sigma, \overline{\Sigma})$ is the \emph{labelled Schreier graph} of the homogeneous $\mathbb{B}_3$-space $\Sigma$ with respect to the generators $\sigma_2^{-1}\sigma_1^{-1}$ and $\sigma_1\sigma_2\sigma_1$ of $\mathbb{B}_3$. We recall the labeling of the Schreier graph of the covering $(\pi, \Sigma, \overline{\Sigma})$ from \cite{22}.
\begin{rem}\label{2.5.3.}
Let $(\pi, \Sigma, \overline{\Sigma})$ be a covering of a $\mathbf{PSL}(2,\mathbb{Z})$-space. Since $x=\sigma_2^{-1}\sigma_1^{-1}\mathbb{B}_3$ and $y=\sigma_1\sigma_2\sigma_1\mathbb{B}_3$, the generators  $\sigma_2^{-1}\sigma_1^{-1}$ and $\sigma_1\sigma_2\sigma_1$ of $\mathbb{B}_3$ correspond to labeled $x-$ and $y-$edges, respectively, in the labeled Schreier graph. Since the covering is a homogeneous space and the sequence
\begin{center}
  $Z(\mathbb{B}_3)\rightarrow \mathbb{B}_3\rightarrow \mathbf{PSL}(2,\mathbb{Z})$
\end{center}
is exact, the fiber $v[*]$ over any $v\in \overline{\Sigma}$ consists of $\left <\Delta\right>$-orbit. If we fix a point $v[0]$ in the fiber $v[*]$ then all other points of the fiber can be enumerated by $v[i]= \Delta^iv[0]$ for all $i \in \{0, 1, . . . ,N - 1\}$, where $N$ is the size of any fiber. Now by choosing a spanning tree of the Schreier graph of $\overline{\Sigma}$ and the images of $v[0]$ along the
arrows of the spanning tree, one can obtain the images of $v[i]$ for all $i$ since $\Delta$ is central. The remaining arrows $v_i \rightarrow v_j$ (which are not on the spanning tree) in the graph of $\overline{\Sigma}$ then have to obtain labels indicating the index shift in the fiber. For instance, a label $l$ tells that $v_i[k]$ is mapped to $v_j [k + l \pmod N]$ for all $k$. Then, up to the choice of the spanning tree, any covering of $\overline{\Sigma}$ is uniquely determined by the labels of the $x-$ and $y-$edges.
\\[1pt] ${}$ \; Observe that, since $\Delta=(\sigma_1\sigma_2)^3=(\sigma_1\sigma_2\sigma_1)^2$, the sum of the labels in any $x-$triangle is $-1$ and the sum of the two labels of a $y-$edge is $1$. We interpret the $y-$edge as double arrows and put the label of the arrow close to its destination. For any $xy-$cycle (or $yx$-cycle) $C$ in $\overline{\Sigma}$, the label of $C$ is the sum of the labels of $x-$ and $y-$edges of the cycle.
\end{rem}

\paragraph{} Now we recall the following lemmas from \cite{22}, which are easy consequences of Remark \ref{2.5.3.} and the Definition \ref{2.4.5.} of a covering with simply intersecting cycles.

\begin{lem}{\label{2.6.}}
Let $\overline{\Sigma}$ be a finite $PSL(2,\mathbb{Z})$-space and let $(\pi, \Sigma, \overline{\Sigma})$ be a covering of $\overline{\Sigma}$ with simply intersecting cycles. Let $v$ be a vertex of the graph of $\overline{\Sigma}$.
\begin{description}
  \item[\;\;\; (a)] If there exists an $x$-loop on $v$ with label $a$ then $3a \equiv -1 \pmod N.$
  \item[\;\;\; (b)] If there exists a $y$-loop on $v$ with label $a$ then $2a \equiv 1 \pmod N.$
\end{description}

\end{lem}

\begin{lem}{\label{2.7.}}
Let $\overline{\Sigma}$ be a finite $PSL(2,\mathbb{Z})$-space and let $(\pi, \Sigma, \overline{\Sigma})$ be a covering of $\overline{\Sigma}$ with simply intersecting cycles. Let $v\in \overline{\Sigma}$ and $w \in C_{yx}(v) \cap C_{xy}(v)$. If
$w\neq v$, then $(\pi, \Sigma, \overline{\Sigma})$ is not trivial.
\end{lem}

\begin{lem}{\label{2.8.}}
Let $\overline{\Sigma}$ be a finite $PSL(2,\mathbb{Z})$-space and let $(\pi, \Sigma, \overline{\Sigma})$ be a covering of $\overline{\Sigma}$ with simply intersecting cycles. Let $v\in \overline{\Sigma}$ and assume that $xv = v$ or
$yv = v$ and that $PSL(2,Z)v \neq {\{v \}}.$ Then the labels of the $xy$- and $yx$-cycles
containing $v$ are $0$.
\end{lem}

\begin{lem}{\label{2.9.}}
Let $\overline{\Sigma}$ be a finite $PSL(2,\mathbb{Z})$-space and let $(\pi, \Sigma, \overline{\Sigma})$ be a covering of $\overline{\Sigma}$ with simply intersecting cycles. Let $v, w \in \overline{\Sigma}$, $N = |\pi^{-1}(v)|$ and assume
that $v \not\equiv w$ and that $v,w$ are on the same  $xy$- and $yx$-cycle. Let $\lambda$
and $\mu$ be the labels of the $xy$- and $yx$-path from $v$ to $w$, respectively. Then
$\lambda\not\equiv \mu \pmod N.$
\end{lem}

\begin{cor}{\label{2.10.}}
Let $\overline{\Sigma}$ be a finite $PSL(2,\mathbb{Z})$-space and let $(\pi, \Sigma, \overline{\Sigma})$ be a covering of $\overline{\Sigma}$ with simply intersecting cycles. Let $v,w \in \overline{\Sigma}$ and assume that
$v \neq w, xv = v, xw = w$ ( or $yv = v, yw = w$ ) and that $v,w$ are on the same $xy$- and $yx$-cycles. Then $\overline{\Sigma}$ has
no coverings with simply intersecting cycles.
\end{cor}

\begin{proof}
Assume to the contrary that $(\pi, \Sigma, \overline{\Sigma})$ is a covering of $\overline{\Sigma}$ with simply
intersecting cycles. Let $N = |\pi^{-1}(v)|$ and let $a$ and $b$ be the labels of the
$x$-loops at $v$ and $w$, respectively. By Lemma \ref{2.6.} we have $3a \equiv -1 \pmod N$ and $3b \equiv -1 \pmod N$ and hence $a \equiv b \pmod N$. Since $xv = v$ and $xw = w$, $v$ and $w$ are on the same $xy$- and $yx$-cycles, and the $xy$ and
$yx$-paths from $v$ to $w$ have the same labels. This is a contradiction to Lemma \ref{2.9.}.
\end{proof}

\begin{center}
\section{\textsc{Cellular Automaton on Hurwitz Orbits}} \label{3}
\end{center}

\paragraph{} In this section we recall the study of a cellular automaton on Hurwitz orbits from \cite{22}. First we recall from \cite{22} the following definition of a cellular automaton on homogeneous $G$-sets which is motivated by a similar definition of cellular automaton on groups in \cite{9}.

\begin{defn}
Let $G$ be a group acting transitively on a set $\Omega$ and let $A$ be a set called an \emph{alphabet}. Let $A^\Omega$ be the set of all functions from $\Omega$ to $A$. Let $S$ be a set, let $(g_s)_{s\in S}$ be a family of elements in $G$, and let $\mu : A^S \rightarrow A$ be a map. Then the map $\tau : A^\Omega \rightarrow A^\Omega$ such that
\begin{center}
$\tau(f)(w) = \mu((f(g _s. w))_{s\in S})$
\end{center}
for all $f \in A^\Omega , w\in \Omega$, is called a cellular automaton over $(G,\Omega)$ with alphabet $A$.
\end{defn}

\paragraph{} A good interpretation of a cellular automaton over $(G,\Omega)$ is the following. For any $w\in \Omega$, consider the family of points $(g_s)_{s\in S}$ as the neighborhood of $w$. Then for any function $f \in A^\Omega$, the value of $\tau(f)$ at $w$ is obtained from the values of $f$ in the neighborhood of $w$ according to the local defining rule determined by $\mu$. Note that the cellular automata to be considered here are with the alphabet $A = \mathbb{Z}_2$. For any function $f \in \mathbb{Z}_2^\Omega$ let $supp \; f= \{w\in \Omega |f(w)=1\}$, and the characteristic function of a set $I\subseteq \Omega$ is
  \[ \chi_I \in \mathbb{Z}_2^\Omega, \; w \rightarrow \left\{
  \begin{array}{l l}
    1 & \text{if $w\in I$},\\
    0 & \text{if otherwise}.
  \end{array} \right.\]

\begin{defn}
Let $\tau$ be a cellular automaton over $(G,\Omega)$ with alphabet $\mathbb{Z}_2$. Then $\tau$ is said to be monotonic if
\begin{description}
  \item[(1)] $supp \; f \subseteq supp \; \tau(f)$ for all $f \in \mathbb{Z}_2^\Omega$, and
  \item[(2)] $supp \; \tau(f) \subseteq supp \; \tau(g)$ for all $f, g \in \mathbb{Z}_2^\Omega$ with $supp \; f \subseteq supp \; g$ \end{description}
\end{defn}

\begin{defn}
Let  $\tau$ be a monotonic cellular automaton over $(G,\Omega)$ with alphabet $\mathbb{Z}_2$. For any subsets $I, J \subseteq \Omega$ with $I\subseteq J$, the subset $I$ is said to \emph{spread} to $J$, if $J \subseteq (\tau^n (\chi_I))$ for some $n \in \mathbb{N}$. A subset $I \subseteq \Omega$ is a \emph{quarantine} if $\tau (\chi_I) = \chi_I$ . A subset $I \subseteq \Omega$ is called a \emph{plague} if the smallest quarantine containing $I$ is $\Omega$.

\end{defn}

\paragraph{} Note that if a subset $I$ spreads to another subset $J$ of $\Omega$, then any subset $I^\prime \subseteq \Omega$ with $I\subseteq I^\prime$ spreads to $J$. Assume that $\Omega$ has only finitely many points. Then a subset $I$ of $\Omega$ is a plague if and only if it spreads to $\Omega$. In this case, any subset of $\Omega$ containing $I$ is a plague.

\paragraph{} Now we recall from \cite{22} the examples of cellular automata over $(G,\Omega)=(\mathbb{Z}, \mathbb{Z}_m)$ for $m\in \mathbb{N}_{\geq 2}$ and $(G,\Omega)=(\mathbb{B}_3, \Sigma)$.

\begin{exa}{\label{4.4.}}
Let $f\in \mathbb{Z}_2^{\mathbb{Z}_m}$, $r\in \mathbb{N}$, and $a_1,...,a_r \in \mathbb{Z}_m\setminus \{0\}$. Let $A=\mathbb{Z}_2$, $S=\{0,1,...,r\}$, and $(g_s)_{s\in S}=(0,-a_1,-a_2,...,-a_r)\in G^S$. Define $\mu : A^S\rightarrow A$ by
  \[ \mu(f_0,f_1,...,f_7)= \left\{
  \begin{array}{l l}
    1 & \text{if $f_0=1$ or $f_1=f_2=...=f_r=1$},\\
    0 & \text{if otherwise}.
  \end{array} \right.\]
The map $\tau: \mathbb{Z}_2^{\mathbb{Z}_m}\rightarrow \mathbb{Z}_2^{\mathbb{Z}_m}$ defined by $\mu$ and $(g_s)_{s\in S}$ is then a monotonic cellular automaton. By definition, $supp \; \tau(f) \subseteq supp \; f\cup \{w\in \Omega|f(x-a_1)=...=f(x-a_1)=1\}$ for all $f\in \mathbb{Z}_2^{\mathbb{Z}_m}$. The plagues for special cases of the cellular automaton over $(\mathbb{Z}, \mathbb{Z}_m)$, which are also used to calculate plagues on the Hurwitz orbits, are the following.
\begin{description}
\item[Case 1.] Let $r=1$ and $a_1=\lambda$. The cellular automaton is determined by the rule
 \begin{center}
    $supp \; \tau(f) \subseteq supp \; f\cup \{x\in \mathbb{Z}_m|f(x-\lambda)=1\}$.
  \end{center}
${}$ \; \; Let $\Gamma=\left < \lambda\right>$ and let $I$ be a set of representatives for $\Omega/\Gamma$. Then $I$ is a plague.
\item[Case 2.] Let $\lambda \in \mathbb{Z}_m$, $r=3$, $a_1=1, a_2=\lambda + 1$, and $a_3=-\lambda$. Let $\Gamma=\left < \lambda\right>$ and let $I$ be the union of a
\\ ${}$\;\;\;\;\; set of representatives for $\Omega/\Gamma$ with $\Gamma$. For example, $I=\left < \lambda\right>\cup \{1,2,...,\lambda-1\}$. Now if $supp \; f$
\\ ${}$\;\;\;\;\; contains a coset $a+\Gamma$, where $a+1\in I$, then $supp \; f$ spreads to $a+1+\Gamma$. Thus $I$ is a plague.
  \item[Case 3.] Let $\lambda \in \Omega\setminus \{0,1\}$, $r=2, a_1=\lambda, a_2=\lambda -1$. Let $I=\{0,1,...,(m-1)/2\}$ if $m$ is odd, and
  \\ ${}$\;\;\;\;\; $I=\{0,1,...,(m/2)-1\}$ if $m$ is even. Then $I$ is a plague of size $\leq (m+1)/2$. It is in general
  \\ ${}$\;\;\;\;\; not minimal, for example for $m\geq 3$, $\lambda=2$ the set $\{0, 1\}$ is a plague.
\end{description}
\end{exa}
\begin{exa}{\label{4.5.}} Let $G=\mathbb{B}_3$ and $\Omega=\Sigma\subseteq R^3$ be an Hurwitz orbit. Take $A = \mathbb{Z}_2$, $S=\{1,2,...,7\}$ and
\begin{center}
  $(g_s)_{s\in S}=(1, \sigma_2, \sigma_1\sigma_2,\sigma_2^{-1}\sigma_1^{-1},\sigma_1^{-1},\sigma_2^{-1},\sigma_1)\in \mathbb{B}_3^7$.
\end{center}
Consider the neighborhood of $x_1\in \Sigma$ given in Figure \ref{Figure 1.}, where the solid arrow indicates the action of $\sigma_1$, the dashed arrow indicates the
action of $\sigma_2$, and $x_s = g_s.x_1$ for all $s \in S$.

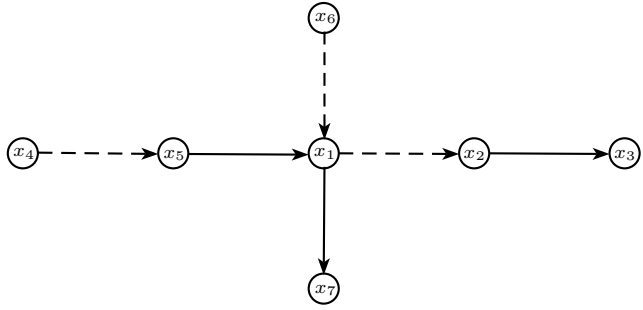
\begin{figure}[hb]
\centering

\psset{xunit=1.0cm,yunit=.9cm,algebraic=true,dimen=middle,dotstyle=o,dotsize=3pt 0,linewidth=0.8pt,arrowsize=3pt 2,arrowinset=0.25}
\begin{pspicture*}(-4.54,-2.4)(4.76,2.4)
\pscircle(0,0){0.2}
\pscircle(2,0){0.2}
\pscircle(4,0){0.2}
\pscircle(-2,0){0.2}
\pscircle(-4,0){0.2}
\pscircle(0,2){0.2}
\pscircle(0,-2){0.2}
\psline[linestyle=dashed]{->}(0.2,0)(1.8,-0.01)
\psline{->}(2.2,0)(3.8,-0.01)
\psline[linestyle=dashed]{->}(0.01,1.8)(0.01,0.2)
\psline{->}(-1.8,-0.01)(-0.2,-0.02)
\psline{->}(0.01,-0.2)(0,-1.8)
\psline[linestyle=dashed]{->}(-3.8,0.01)(-2.2,-0.01)
\begin{scriptsize}
\rput[bl](-0.13,-0.08){$x_1$}
\rput[bl](1.86,-0.1){$x_2$}
\rput[bl](3.86,-0.1){$x_3$}
\rput[bl](-2.13,-0.10){$x_5$}
\rput[bl](-4.13,-0.08){$x_4$}
\rput[bl](-0.12,1.92){$x_6$}
\rput[bl](-0.12,-2.1){$x_7$}
\end{scriptsize}
\end{pspicture*}
\caption{Neighbors of $x_1$} \label{Figure 1.}
\end{figure}

Define $\mu : A^7\rightarrow A$ by
\begin{center}
  $\mu(f_1,f_2,...,f_7)=f_1\vee f_2f_3\vee f_4f_5\vee f_6f_7=1-(1-f_1)(1-f_2f_3)(1-f_4f_5)(1-f_6f_7)$,
\end{center}
where $f_1,f_2,...,f_7\in A$, and $\vee$ denotes logical or. Then the map $\tau$ defined by $\mu$ and $(g_s)_{s\in S}$ is a monotonic cellular automaton over $(\mathbb{B}_3, \Sigma)$. A plague of this cellular automaton is literally the same which is defined in \cite{21} as follows.
\end{exa}
\begin{defn} \label{3.1.}
 A \emph{quarantine} of an Hurwitz orbit $\Sigma$ is a non-empty subset $Q\subseteq \Sigma$ such that if any two of elements $(r, s, t), \sigma_2(r, s, t)$, and $\sigma_1\sigma_2(r, s, t)$ are in $Q$, then the third one is also in $Q$. A non-empty subset $P$ of an Hurwitz orbit $\Sigma$ is called \emph{plague} if the smallest quarantine of $\Sigma$ containing $P$ is $\Sigma$.
\end{defn}

\begin{rem}
Note that the graph theoretical structure of plague on the Hurwitz orbits is closely related to the graph bootstrap percolation (see \cite{4}). The principle of the method to calculate the plague on the Hurwitz orbits in the language of cellular automata is formulated in \cite{22} as follows.
\\[1pt] ${}$ \; Consider $\sigma_2^{-1}\sigma_1^{-1}$ and $\sigma_1\sigma_2\sigma_1$ as generators of $\mathbb{B}_3$. Let $\tau$ be a cellular automaton over $(\mathbb{B}_3, \Sigma)$. Let $f$ be a $\mathbb{Z}_2-$valued function on $\Sigma$ and let $P=supp \; f$. Let $v$ be a point in the Hurwitz orbit quotient $\overline{\Sigma}$ and $v[*]$ be a fiber over $v$ of size $N$. Let $I$ be a subset of $\mathbb{Z}_N$ and let $v[I]=\{v[i]|i \in I\}$ be the corresponding subset of $v[*]$. Now consider the following three neighboring subsets of $v[I]$
\begin{center}
  $(\sigma_2^{-1}\sigma_1^{-1})^{-1}. v[I]$, \; $\sigma_1\sigma_2\sigma_1. v[I]$, and \; $(\sigma_2^{-1}\sigma_1^{-1}). v[I]$.
\end{center}
These three subset are denoted in Figure \ref{Figure 2.} by $v_1[I-c]$, $v_2[I+a]$ and $v_3[I+b+1]$, where $I+a=\{i+a|i\in I\}$. In this setting $v[I]$ is called a \emph{pivot}. Now by Example \ref{4.4.} if $P=supp \; f$ contains the subsets $v_1[I-c]=\sigma_1\sigma_2. v[I]$ and $\sigma_2.v_1[I-c]=v_2[I+a]$, then $supp \; \tau(f)$ contains $\sigma_1. v_2[I+a]=\sigma_1\sigma_2\sigma_1\sigma_2. v[I]=\sigma_1^{-1}\sigma_2^{-1}\Delta. v[I]=v_3[I+b+1]$. Similarly, if any two of the neighboring subsets $v_1[I-c]$, $v_2[I+a]$, and $v_3[I+b+1]$ of $v[I]$ are contained in $P$, then the third is a subset of $supp \; \tau(f)$. Moreover, $supp \; \tau(f)$ is the smallest subset of $\Sigma$ containing $supp \; f$ and all sets constructed this way for some point $v$ and some subset $I\subseteq \mathbb{Z}_N$.
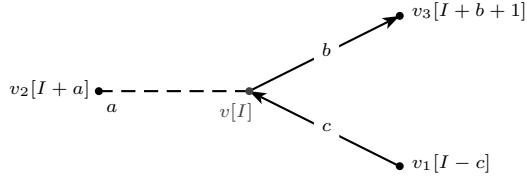
\begin{figure}[hb!]
\centering
\psset{xunit=1.0cm,yunit=1.0cm,algebraic=true,dimen=middle,dotstyle=o,dotsize=3pt 0,linewidth=0.8pt,arrowsize=3pt 2,arrowinset=0.25}
\begin{pspicture*}(-3.2,-1.4)(3.7,1.32)
\psline{->}(0,0)(2,1)
\psline{->}(2,-1)(0,0)
\psline[linestyle=dashed](0,0)(-2,0)
\begin{scriptsize}
\psdots[dotstyle=*,linecolor=darkgray](0,0)
\rput[bl](-0.4,-0.41){\darkgray{$v[I]$}}
\psdots[dotstyle=*](2,1)
\rput[bl](2.16,0.94){$v_3[I+b+1]$}
\psdots[dotstyle=*](2,-1)
\rput[bl](2.16,-1.08){$v_1[I-c]$}
\psdots[dotstyle=*](-2,0)
\rput[bl](-3.19,-0.1){$v_2[I+a]$}
\rput[bl](0.86,0.36){\psframebox*{$b$}}
\rput[bl](0.86,-0.64){\psframebox*{$c$}}
\rput[bl](-1.9,-0.28){$a$}
\end{scriptsize}
\end{pspicture*}

\caption{Neighbors of $v[I]$} \label{Figure 2.}
\end{figure}

\end{rem}

\subsection{Immunity and Weight on Hurwitz Orbits.}
\paragraph{} In this section we recall the definitions of immunity and weight on the Hurwitz orbit $\Sigma$ from \cite{22}. The notation $\Sigma^{N;a,b,c}_{X}$ is used in \cite{22} for the Hurwitz orbit $\Sigma$, where $N$ stands for the fiber size, the numbers $a, b, c$ are the values of the individual labels determining the covering and $X$ stands for any possible index.
\begin{defn}{\label{3.8.}}
Let $P$ is a plague of smallest possible size on a Hurwitz orbit $\Sigma$. Then the \emph{immunity} of $\Sigma$ is defined as the quotient $|P|/|\Sigma|\in {\mathbb{Q}\cap (0, 1]}$. The immunity of $\Sigma$ is denoted by $imm(\Sigma)$. 
\end{defn}
\paragraph{} Note that in the case of braided racks, immunities can be computed manually, because there is only a small number of Hurwitz orbits for braided racks (see \cite{21}, Proposition 9). In the case of arbitrary racks, the immunity on a Hurwitz orbits can be estimated by using the length $i$ of $\sigma_1$-cycle and the length $j$ of $\sigma_2$-cycle of each point of that Hurwitz orbit. For estimating the immunity on a Hurwitz orbit the following matrix is defined in \cite{22}:
\begin{center}
  $(\omega^\prime_{ij})_{i,j\geq1}=\left(
                       \begin{array}{cccccc}
                         1 & 1/3 & 11/24 & 1/2 & 1/2 & \cdots\\
                         1/3  & 1/3  & 1/3  & 1/3  & 1/3 & \cdots \\
                         11/24 & 1/3 & 7/24 & 7/24  & 7/24 & \cdots \\
                         1/2 & 1/3 & 7/24 & 1/4  & 1/4 & \cdots\\
                         1/2 & 1/3 & 7/24  & 1/4 & 1/4 & \cdots\\
                         \vdots & \vdots & \vdots & \vdots & \vdots & \ddots
                       \end{array}
                     \right).$
\end{center}

\begin{defn}{\label{3.9.}}
Let $\Sigma$ be a finite homogeneous $\mathbb{B}_3$-space and $v \in \Sigma$ such that $v$ belongs to a $\sigma_1$-cycle of length $i$, and also to a $\sigma_2$-cycle of length $j$. Let $\omega:\Sigma \longrightarrow \mathbb{Q}$ be the map defined by

\[ \omega(v) = \left\{
  \begin{array}{l l}
    \omega^\prime_{ij}+ \frac{1}{30}=\frac{13}{40} & \text{if $\Sigma=\Sigma^{5;3,2}_{4A}$ and $v\in v_1[*]$},
    \\ \omega^\prime_{ij} +\frac{1}{12}=\frac{1}{3} & \text{if $\Sigma=\Sigma^{4;2,2}_{6A}$ and $v\in v_3[*]$},
    \\\omega^\prime_{ij} + \frac{1}{24}=\frac{1}{3}& \text{if $\Sigma$ is the trivial covering of $\Sigma_{12C}$},
    \\\omega^\prime_{ij}& \text{otherwise}.
  \end{array} \right.\]
The \emph{weight} of $\Sigma$ is defined as $\omega(\Sigma)=\frac{1}{|\Sigma|}\sum\limits_{v\in \Sigma}\omega(v)$.
\end{defn}

\begin{exa}
Let $R=\{r_1, r_2, r_3\}$ be the conjugacy class in the symmetric group $S_3$ with $r_1=(2\;3)$, $r_2=(1\;3)$, $r_3=(1\;2)$. Let $\Sigma=\{A, B, C, D, E, F, G, H\}$ be the Hurwitz orbit of $(r_1,r_2,r_3)$. The Hurwitz orbit quotient of $\Sigma$ is $\overline{\Sigma}=\{v_1,v_2,v_3,v_4\}$ with $v_1=\{D,E\},v_2=\{B,G\},v_3=\{A,H\},v_4=\{C,F\}$, as shown in Figure \ref{Figure 3.}.
\begin{figure}[hb]
\centering
\psset{xunit=.8cm,yunit=.7cm,algebraic=true,dimen=middle,dotstyle=o,dotsize=3pt 0,linewidth=0.8pt,arrowsize=3pt 2,arrowinset=0.25}
\begin{pspicture*}(-3.6179829653912146,-3.46452597399095)(11.194666857186103,4.0266695766027487)
\psline{->}(-2.,2.)(0.,2.)
\psline{->}(0.,2.)(-2.,0.)
\psline{->}(-2.,0.)(-2.,2.)
\psline{->}(0.,-2.)(2.,0.)
\psline{->}(2.,0.)(2.,-2.)
\psline{->}(2.,-2.)(0.,-2.)
\psline[linestyle=dashed]{->}(0.,2.)(2.,2.)
\psline[linestyle=dashed]{->}(2.,2.)(2.,0.)
\psline[linestyle=dashed]{->}(2.,0.)(0.,2.)
\psline[linestyle=dashed]{->}(-2.,0.)(0.,-2.)
\psline[linestyle=dashed]{->}(0.,-2.)(-2.,-2.)
\psline[linestyle=dashed]{->}(-2.,-2.)(-2.,0.)
\psline{->}(3.,0.)(4.,0.)

\rput[bl](3.344045510305697,0.1197178042908904){$\pi$}
\psline{->}(8.,0.)(10.,2.)
\psline{->}(10.,2.)(10.,-2.)
\psline{->}(10.,-2.)(8.,0.)
\psline[linestyle=dashed](6.,0.)(8.,0.)
\parametricplot[linestyle=dashed]{-0.7853981633974483}{0.7853981633974483}{1.*2.8284271247461903*cos(t)+0.*2.8284271247461903*sin(t)+8.|0.*2.8284271247461903*cos(t)+1.*2.8284271247461903*sin(t)+0.}
\begin{scriptsize}
\psdots[dotstyle=*](-2.,2.)
\rput[bl](-2.2847018704764803,2.1379199901634545){$A$}
\psdots[dotstyle=*](0.,2.)
\rput[bl](-0.09985913714655614,2.1564356065476065){$B$}
\psdots[dotstyle=*](2.,2.)
\rput[bl](2.0664679797992163,2.1194043737793025){$C$}
\psdots[dotstyle=*](-2.,0.)
\rput[bl](-2.395795568781392,-0.04692274316647725){$D$}
\psdots[dotstyle=*](2.,0.)
\rput[bl](2.140530445335824,-0.009891510398173327){$E$}
\psdots[dotstyle=*](-2.,-2.)
\rput[bl](-2.31921237885557208,-2.32687967092647128){$F$}
\psdots[dotstyle=*](0.,-2.)
\rput[bl](-0.08134352076240425,-2.32687967092647128){$G$}
\psdots[dotstyle=*](2.,-2.)
\rput[bl](2.0849835961833683,-2.2502810928805608){$H$}
\psdots[dotstyle=*](6.,0.)

\pnode(-2,2){A}
\nccircle[angleA=45, linestyle=dashed]{<-}{A}{.5cm}

\pnode(2,2){A}
\nccircle[angleA=-45]{<-}{A}{.5cm}

\pnode(-2,-2){A}
\nccircle[angleA=-225]{<-}{A}{.5cm}

\pnode(2,-2){A}
\nccircle[angleA=225, linestyle=dashed]{<-}{A}{.5cm}

\rput[bl](6.084356735160179,-0.2505945233921488){{$v_1$}}
\psdots[dotstyle=*](8.,0.)
\rput[bl](7.687793442502153,-0.26911013977630077){{$v_2$}}
\psdots[dotstyle=*](10.,2.)
\rput[bl](9.917089326679621,2.211982455700062){{$v_3$}}
\psdots[dotstyle=*](10.,-2.)
\rput[bl](9.917089326679621,-2.2873123256488648){{$v_4$}}
\rput[bl](8.661699192782963,0.6973736924252729){\psframebox*{$-1$}}
\rput[bl](10.2509589050972483,1.804158109511815278){{$1-b$}}
\rput[bl](7.735918373575368,0.06417095513843453){{$1$}}
\rput[bl](6.065841118776026,0.1197178042908904){{$0$}}
\rput[bl](10.39198365661546,-2.03245002736729627){$b$}
\pnode(6,0){A}
\nccircle[angleA=90]{<-}{A}{.5cm}
\rput[bl](4.5,0.){{$a$}}
\end{scriptsize}
\end{pspicture*}
\caption{The Covering $(\pi,\Sigma,\overline{\Sigma})$} \label{Figure 3.}
\end{figure}
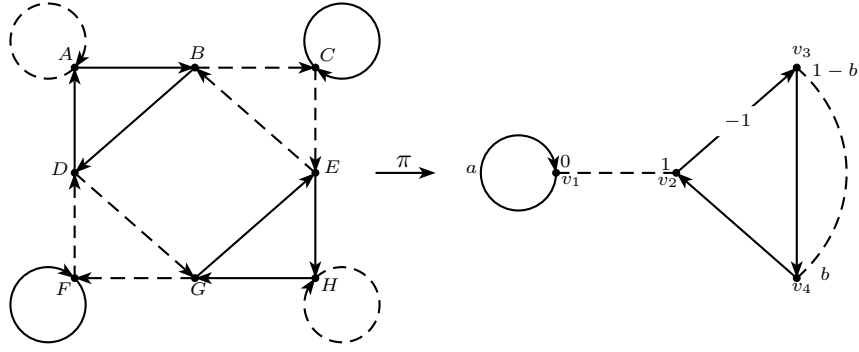
In this covering $N=2$. The $xy-$cycles with labels are: $(v_1\;v_3\;v_2)$ with $a+b$ and $(v_4)$ with $1-b$. The $yx-$cycles with labels are: $(v_1\;v_2\;v_4)$ with $a+b$ and $(v_3)$ with $1-b$. By Lemma \ref{2.6.} and Lemma \ref{2.9.}, it follows that $3a \equiv -1 \pmod N$ and $a+b=0 \pmod N$, and hence $a=b=1$. Note that for $i\in \{0,1\}$, $v_1[i], v_2[i]$ have two 3-cycles and $v_3[i], v_4[i]$ have cycles of lengths $1$ and $3$. Hence for $N=2$ we have $\omega(\Sigma)=\frac{1}{|\Sigma|}\sum\limits_{v\in \Sigma}\omega(v)=1/8(2N(7/24)+2N(11/24))=3/8$.
Now by the following table $P=\{v_1[*],v_3[0]\}=\{v_1[0],v_1[1],v_3[0]\}=\{A, D, E\}$ is a plague on $\Sigma$.

\begin{center}
  \begin{tabular}{c|ccc}
  pivot & $v_1[*]$ & $v_4[0]$ & $v_2[*]$  \\
  \hline
   & $v_2[*]$ & $v_3[1]$ & $v_4[*]$ \\
\end{tabular}
\end{center}
Hence $imm(\Sigma)=3/8 = \omega(\Sigma)$.

\end{exa}

\paragraph{} The weight of an Hurwitz orbit provides a good upper bound for the immunity of that Hurwitz orbit which is given in \cite{22}(Theorem 6.3). We recall this theorem here.

\begin{thm}{\label{2.7.2.}}
Let $\Sigma$ be a covering with simply intersecting cycles of finite homogeneous $PSL(2,\mathbb{Z})$-space $\overline{\Sigma}$. Assume that any $xy$-cycle of $\overline{\Sigma}$ has at most four elements. Then $imm(\Sigma)\leq \omega(\Sigma)$.

\end{thm}
\begin{rem}
Note that by Definition \ref{3.9.} it follows that $\omega(\Sigma)\geq 1/4$. Therefore if $imm(\Sigma)\leq 1/4$ then $imm(\Sigma)\leq \omega(\Sigma)$. Note also that the assumption about the length of an $xy-$cycle in Theorem \ref{2.7.2.} can be replaced by a weaker assumption that $imm(\Sigma)\leq \omega(\Sigma)$ for all homogeneous $B_3$-spaces $\Sigma$. This weaker assumption is proposed as a Conjecture \ref{1.1.} in \cite{22}. Note that by Theorem \ref{2.7.2.}, the Conjecture \ref{1.1.} is true for the coverings $\Sigma$ of a finite homogeneous $PSL(2,\mathbb{Z})$-space $\overline{\Sigma}$ such that any $xy$-cycle of $\overline{\Sigma}$ has at most four elements. In Section \ref{5} we will prove the Conjecture \ref{1.1.} for any covering with simply intersecting cycles of finite homogeneous $PSL(2,\mathbb{Z})$-spaces $\overline{\Sigma}$ whose pointed Schreier graph $\mathcal{G}$ is with $V_x\neq \emptyset$ and $V_{xy}=\emptyset$. More precisely, we prove the following theorem by using the robust subgraphs of the pointed Schreier graph $\mathcal{G}$ which we define in the next section.
\end{rem}

\begin{thm}{\label{3.13.}}
Let $\Sigma$ be a covering with simply intersecting cycles of finite homogeneous $\mathbf{PSL}(2,\mathbb{Z})$-space $\overline{\Sigma}$. Assume that the pointed Schreier graph of $\overline{\Sigma}$ is $\mathcal{G}$ with $V_x\neq \emptyset$ and $V_{xy}=\emptyset$. Then $imm(\Sigma)\leq \omega(\Sigma)$.
\end{thm}

\begin{center}
\section{\textsc{Robust Subgraphs of Pointed Schreier Graph of the Hurwitz Orbit Quotients}} \label{4}
\end{center}
\paragraph{} Let $t,j$ be integers with $0\leq j \leq t$. Let $\mathcal{G}$ be a finite pointed Schreier graph of size $n$ with $V_x\neq \emptyset$. Let $t$ be the number of triangles of $\mathcal{G}$ and $v_0 \in V_x$ be a distinguished vertex in $\mathcal{G}$. Let $\mathcal{H}_j$ be a connected subgraph of $\mathcal{G}$ with $j$ triangles such that:
\begin{itemize}
  \item $v_0$ belongs to $\mathcal{H}_j$,
  \item each $x-$edge of $\mathcal{H}_j$, which is not an $x-$loop, belongs to a triangle in $\mathcal{H}_j$,
  \item each vertex of $\mathcal{H}_j$ is either adjacent to itself through a $y-$loop or adjacent to another vertex of $\mathcal{H}_j$ through $y-$edge.
\end{itemize}
We call this subgraph $\mathcal{H}_j$ a \emph{robust subgraph} of $\mathcal{G}$. Note that $\mathcal{H}_t=\mathcal{G}$. If $\mathcal{H}_0=\mathcal{G}$ and $n=1$, then $v_0$ is with both $x-$ and $y-$loops. If $\mathcal{H}_0 \neq \mathcal{G}$ and $\mathcal{G}$ is with $|V_x|=1$ and $V_{y}=\emptyset=V_{xy}$ then the robust subgraphs $\mathcal{H}_0$ and $\mathcal{H}_1$ are shown in Figure \ref{Figure 4.}.

\begin{figure}[hb!]
\centering

\psset{xunit=1.0cm,yunit=.8cm,algebraic=true,dimen=middle,dotstyle=o,dotsize=3pt 0,linewidth=0.8pt,arrowsize=3pt 2,arrowinset=0.25}
\begin{pspicture*}(0.35,-1.31)(9.12,1.24)
\psline(6,0)(7,1)
\psline(7,1)(7,-1)
\psline(7,-1)(6,0)
\psline[linestyle=dashed](7,1)(8,1)
\psline[linestyle=dashed](7,-1)(8,-1)
\psline[linestyle=dashed](6,0)(5,0)
\psline[linestyle=dashed](2,0)(3,0)
\begin{scriptsize}
\psdots[dotstyle=*](6,0)
\psdots[dotstyle=*](7,1)
\psdots[dotstyle=*](7,-1)
\psdots[dotstyle=*](5,0)
\rput[bl](5.05,-0.3){$v_0$}
\psdots[dotstyle=*](8,1)
\psdots[dotstyle=*](8,-1)
\psdots[dotstyle=*](2,0)
\rput[bl](2.05,-0.3){$v_0$}
\psdots[dotstyle=*](3,0)
\end{scriptsize}
\end{pspicture*}

\pnode(-2.73,1.82){A}
\nccircle[angleA=90]{<-}{A}{.5cm}

\pnode(0.26,2.35){A}
\nccircle[angleA=90]{<-}{A}{.5cm}

\caption{Small Robust Subgraphs $\mathcal{H}_0$ and $\mathcal{H}_1$} \label{Figure 4.}
\end{figure}
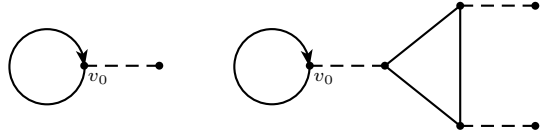

\paragraph{} Let $\mathcal{H}$ be the family of all robust subgraphs $\mathcal{H}_j$ of a finite $\mathbf{PSL}(2,\mathbb{Z})$-space with $V_x\neq \emptyset$. Then $\mathcal{\mathcal{H}}$ is a finite partially ordered set by subgraph inclusion relation, $\mathcal{H}_i\prec \mathcal{H}_j$, for all non-negative integers $i,j$ with $0\leq i \leq j \leq t$. Let $e(\mathcal{H}_j)$ denote the number of $y-$edges of $\mathcal{H}_j$. Note that we will consider the $y-$loop at a vertex of $\mathcal{G}$ as one $y-$edge. If $i,j\in \{0,1,...,t\}$ with $\mathcal{H}_i \prec \mathcal{H}_j$ and $m_j\in \{0,1,2\}$ such that $e(\mathcal{H}_j)=e(\mathcal{H}_i)+m_j$, then we write $\mathcal{H}_i\prec_{m_j} \mathcal{H}_j$. Since $m_j$ varies with $j$, we write a sequence of robust subgraphs in $\mathcal{\mathcal{H}}$ as $\mathcal{H}_0\prec_{m_1} \mathcal{H}_1\prec_{m_2}...\prec_{m_t} \mathcal{H}_t$, where $m_j\in \{0,1,2\}$ for $1\leq j\leq t$. Observe that if $|V(\mathcal{H}_j)|=n_j$ is the size of a robust subgraph $\mathcal{H}_j$ and $\mathcal{H}_i\prec_{m_j} \mathcal{H}_j$ then $n_j=n_i+m$, for $m\in \{0, 1, 2,3,4\}$. For example in Figure \ref{Figure 4.} we have $\mathcal{H}_0\prec_{2} \mathcal{H}_1$ and $n_1=n_0+m=2+4=6$.
\\[1pt] ${}$ \; Let $\mathcal{G}$ be a a pointed Schreier graph of a finite $\mathbf{PSL}(2,\mathbb{Z})$-space with $t$ triangles.  Let $\mathcal{H}_j$ be a robust subgraph of $\mathcal{G}$ for $j\in \{0,1,...,t\}$. We define a \emph{fragment} $\mathcal{F}$ of $\mathcal{H}_j$ as a subgraph of $\mathcal{H}_j$ which is separated from $\mathcal{H}_j$ by the robust subgraph $\mathcal{H}_i$ for any $i\in \{0,1,...,t-1\}$ with $i=j-1$. That is, $\mathcal{F}=\mathcal{H}_j\setminus \mathcal{H}_i$ for any $j\in \{0,1,...,t\}$ and $i=j-1$.

\subsection{Coverings of Robust Subgraphs with Plague and Immunity.}
\paragraph{} Let $\Sigma$ be a covering with simply intersecting cycles of finite homogeneous $\mathbf{PSL}(2,\mathbb{Z})$-space $\overline{\Sigma}$ of size $n$ with $V_x\neq \emptyset$. Let $\mathcal{G}$ be the pointed Schreier graph of $\overline{\Sigma}$ with distinguished vertex $v_0 \in V_x$ and $\mathcal{H}_j$ be a robust subgraph of $\mathcal{G}$. For the map $\pi:\Sigma \rightarrow V(\mathcal{G})$, the covering of $\mathcal{H}_j$ is the set
\begin{center}
$\Sigma_{\mathcal{H}_j}:=\pi^{-1}(V(\mathcal{H}_j))\subseteq \Sigma$,
\end{center}
such that $|\Sigma_{\mathcal{H}_j}|=|V(\mathcal{H}_j)|N=n_jN$. We denote the plague and immunity on $\Sigma_{\mathcal{H}_j}$ by $P(\Sigma_{\mathcal{H}_j})$ and $imm(\Sigma_{\mathcal{H}_j})$ respectively. If $P(\Sigma_{\mathcal{H}_j})$ consists of complete fibers over $p_j$ vertices of $\mathcal{H}_j$ then $|P(\Sigma_{\mathcal{H}_j})|:=p_jN$ and $imm(\Sigma_{\mathcal{H}_j})\leq \frac{|P(\Sigma_{\mathcal{H}_j})|}{|\Sigma_{\mathcal{H}_j}|}= \frac{p_jN}{n_jN}= \frac{p_j}{n_j}$.

\paragraph{} Let $J$ be a set of subscripts $j$ of all robust subgraphs $\mathcal{H}_j$ of $\mathcal{G}$ and $K:=\{k\in J:\mathcal{H}_{k-1}\prec_2 \mathcal{H}_k\}$. Note that $0, t \notin K$, where $t$ is the total number of triangles of $\mathcal{G}$. Let $v:K\rightarrow V(\mathcal{G})$ be a map such that $v(k)\in \mathcal{H}_k\backslash \mathcal{H}_{k-1}$ for any $k\in K$. Then by using the set $P_K:=\{v(k)[*]|k\in K \}\subset \Sigma$ we prove the following lemma.

\begin{lem} {\label{4.1.}}
Let $\Sigma$ be a covering with simply intersecting cycles of finite homogeneous $\mathbf{PSL}(2,\mathbb{Z})$-space with $V_x\neq \emptyset$ and $V_{xy}=\emptyset$. Then $v_0[*] \cup P_K$ is a plague on $\Sigma_{\mathcal{H}_j}$ for all $1\leq j\leq t$, where $v_0\in V_x$.
\end{lem}

\begin{proof}
We prove the claim by induction on $1\leq j\leq t$. For $j=1$ we have $\mathcal{H}_0\prec_2 \mathcal{H}_1$, since $V_x\neq \emptyset$ and $V_{xy}=\emptyset$. Now consider the robust subgraph $\mathcal{H}_1$ shown in Figure \ref{Figure 5.}. Then for $\mathcal{H}_0\prec_2 \mathcal{H}_1$ we have $j=1$, $K=\{1\}$ and $v(1)\in \{v_2, v_3, v_4, v_5\}$. If we take $v(1)=v_2$, then $P_K=\{v_2[*]\}$. Now by the following table $v_0[*] \cup P_K=\{v_0[*], v_2[*]\}$ is a plague on $\Sigma_{\mathcal{H}_1}$. \begin{center}
  \begin{tabular}{c|cccc}
  pivot & $v_0[*]$ & $v_1[*]$ & $v_3[*]$ & $v_2[*]$ \\
  \hline
   & $v_1[*]$ & $v_3[*]$ & $v_4[*]$ & $v_5[*]$\\
\end{tabular}
\end{center}
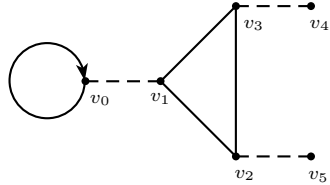
\begin{figure}[hb]
\centering

\psset{xunit=1.0cm,yunit=1.0cm,algebraic=true,dimen=middle,dotstyle=o,dotsize=3pt 0,linewidth=0.8pt,arrowsize=3pt 2,arrowinset=0.25}
\begin{pspicture*}(-1.84,-1.72)(3.98,1.62)
\psline(1,0)(2,1)
\psline(2,1)(2,-1)
\psline(2,-1)(1,0)
\psline[linestyle=dashed](2,1)(3,1)
\psline[linestyle=dashed](2,-1)(3,-1)
\psline[linestyle=dashed](0,0)(1,0)
\begin{scriptsize}
\psdots[dotstyle=*](1,0)
\rput[bl](0.86,-0.3){$v_1$}
\psdots[dotstyle=*](2,1)
\rput[bl](2.1,0.68){$v_3$}
\psdots[dotstyle=*](2,-1)
\rput[bl](1.98,-1.32){$v_2$}
\psdots[dotstyle=*](3,1)
\rput[bl](2.98,0.68){$v_4$}
\psdots[dotstyle=*](3,-1)
\rput[bl](2.96,-1.34){$v_5$}
\psdots[dotstyle=*](0,0)
\rput[bl](0.06,-0.32){{$v_0$}}
\end{scriptsize}
\end{pspicture*}

\pnode(-1.08,2.15){A}
\nccircle[angleA=90]{<-}{A}{.5cm}

\caption{Robust Subgraphs $\mathcal{H}_1$} \label{Figure 5.}
\end{figure}

\paragraph{} Next suppose that the claim is true for $2\leq j<t$ and $v_0[*] \cup P_K$ is a plague on $\Sigma_{\mathcal{H}_j}$. Then for $j+1$, we have $\mathcal{H}_j\prec_{m_{j+1}} \mathcal{H}_{j+1}$ with $m_{j+1}\in \{0,1,2\}$. First, for $m_{j+1}=0$, $\mathcal{H}_{j+1}=\mathcal{H}_t=G$ and $\mathcal{H}_{j+1}\setminus \mathcal{H}_j=\emptyset$ and $v_0[*] \cup P_K$ spreads to $\Sigma_{\mathcal{H}_{j+1}}$.  Next, for $m_{j+1}=1$, $n_{j+1}=n_j+1$ or $n_{j+1}=n_j+2$. That is, $\mathcal{H}_{j+1}$ has one or two vertices, namely $v_i$ and $v_{i+1}$ which are not on $\mathcal{H}_j$. In this case $v_0[*] \cup P_K$ will also spread to the fibers $v_i[*]$ and $v_{i+1}[*]$. Finally, for $m_{j+1}=2$, $n_{j+1}=n_j+3$ or $n_{j+1}=n_j+4$. That is $\mathcal{H}_{j+1}$ has three or four vertices which are not on $\mathcal{H}_j$. In this case $v_0[*] \cup P_K$ can not spread to $\Sigma_{\mathcal{H}_{j+1}}$. However, if we take the fiber over only one $v_i\in \mathcal{H}_j\backslash \mathcal{H}_{j-1}$ then $v_0[*] \cup P_K \cup v_i[*]$ will spread to $\Sigma_{\mathcal{H}_{j+1}}$.
\end{proof}

\paragraph{} Now we prove the following important lemma.

\begin{lem} {\label{4.2.}}
Let $\Sigma$ be a covering with simply intersecting cycles of finite homogeneous $\mathbf{PSL}(2,\mathbb{Z})$-space $\overline{\Sigma}$. Assume that the pointed Schreier graph $\mathcal{G}$ of $\overline{\Sigma}$ is with $t$ triangles and $V_x\neq \emptyset$, $V_{xy}=\emptyset$. Assume that for $1<i< t$ there exists a robust subgraph $\mathcal{H}_{i}$ such that $imm(\Sigma_{\mathcal{H}_i})\leq  1/4$. Then $imm(\Sigma_{\mathcal{H}_{j}})\leq  1/4$ for all $i+1\leq j\leq t$.
\end{lem}

\begin{proof}
Suppose that $imm(\Sigma_{\mathcal{H}_i})\leq  1/4$ for some $1<i< t$. Suppose there exists a plague $P(\Sigma_{\mathcal{H}_i})$ on $\Sigma_{\mathcal{H}_i}$ which consists of complete fibers over $p_i$ vertices of $\mathcal{H}_i$ such that $\frac{p_{i}}{n_{{i}}}\leq \frac{1}{4}$. No for $i+1\leq j\leq t$ we have $\mathcal{H}_i\prec_{m_{j}} \mathcal{H}_{j}$ with $m_{j}\in \{0,1,2\}$ and $n_{j}=n_i+m$ with $m\in \{0,1,2,3,4\}$. For $m_{j}=0, m=0$ we have $p_{j}=p_i$ and $n_{j}=n_i$. In fact in this case $j=t$. This implies that $imm(\Sigma_{\mathcal{H}_{j}})\leq  1/4$. Next suppose that $m_{j}=1$. Then $m\in \{1,2\}$. For $m_{j}=1$ and $m=1$, we have $p_{j}=p_i$ and $n_{j}=n_i+1$. Therefore we have
\begin{center}
$\frac{p_{j}}{n_{{j}}}=\frac{p_{i}}{n_{i}+1}<\frac{p_{{i}}}{n_{{i}}}\leq  1/4$.
\end{center}
This implies that $imm(\Sigma_{\mathcal{H}_{j}})\leq  1/4$. For $m_{j}=1$ and $m=2$ we have $p_{j}=p_i$ and $n_{j}=n_i+2$. In this case we again have $imm(\Sigma_{\mathcal{H}_{j}})\leq  1/4$ since
\begin{center}
$\frac{p_{{j}}}{n_{{j}}}=\frac{p_{i}}{n_{i}+2}<\frac{p_{i}}{n_{i}}\leq  1/4$.
\end{center}
Next suppose that $m_{j}=2$ and $m=3$. Then $V_{y}\cap (\mathcal{H}_{{j}}\setminus \mathcal{H}_{i})\neq \emptyset$. Let $V_{y}\cap (\mathcal{H}_{{j}}\setminus \mathcal{H}_{i})=\{v_{j}\}$ and let $b_{j}$ is the label of $y-$loop on $v_{j}$ as in Figure \ref{Figure 6.}, where the square boxes (like $\square$) may contain arbitrary fragments of the pointed Schreier graph $\mathcal{G}$.
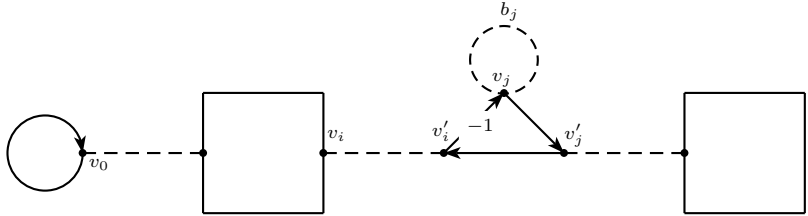
\begin{figure}[hb]
\centering

\psset{xunit=.8cm,yunit=.8cm,algebraic=true,dimen=middle,dotstyle=o,dotsize=3pt 0,linewidth=0.8pt,arrowsize=3pt 2,arrowinset=0.25}
\begin{pspicture*}(-1.74,-2.22)(14.08,2.54)
\psline[linestyle=dashed](0,0)(2,0)
\psline(2,1)(4,1)
\psline(4,1)(4,-1)
\psline(4,-1)(2,-1)
\psline(2,-1)(2,1)
\psline{->}(6,0)(7,1)
\psline{->}(7,1)(8,0)
\psline{->}(8,0)(6,0)
\psline[linestyle=dashed](4,0)(6,0)
\psline[linestyle=dashed](8,0)(10,0)
\pscircle[linestyle=dashed](7,1.55){0.45}
\psline(10,1)(10,-1)
\psline(10,1)(12,1)
\psline(12,1)(12,-1)
\psline(12,-1)(10,-1)
\begin{scriptsize}
\psdots[dotstyle=*](0,0)
\rput[bl](0.1,-0.28){$v_0$}
\psdots[dotstyle=*](2,0)
\psdots[dotstyle=*](4,0)
\rput[bl](4.06,0.2){$v_i$}
\psdots[dotstyle=*](6,0)
\rput[bl](5.8,0.2){$v^\prime_{i}$}
\psdots[dotstyle=*](7,1)
\rput[bl](6.25,0.2) {\psframebox*{$-1$}}
\rput[bl](6.8,1.1){$v_{j}$}
\psdots[dotstyle=*](8,0)
\rput[bl](8.0,0.1){$v^\prime_{j}$}
\psdots[dotstyle=*](10,0)
\rput[bl](6.94,2.2){$b_{j}$}

\pnode(0,0){A}
\nccircle[angleA=90]{<-}{A}{.5cm}
\end{scriptsize}
\end{pspicture*}
\caption{Robust Subgraph $\mathcal{H}_j$} \label{Figure 6.}
\end{figure}
Note that $N>1$ because one can see that there exist $xy-$ and $yx-$cycles which both contain vertices $v^\prime_{i}$ and $v^\prime_{j}$. Also $2b_{j} \equiv 1 \pmod N$, and $b_{j}\neq 0,1$ since $N>1$. Now from Figure \ref{Figure 6.} we have the following table

\begin{center}
\begin{tabular}{c|cc}
  pivot & $v_{j}[b_{j}]$ & $v^\prime_{i}[1+b_{j}]$ \\
  \hline
   & $v^\prime_{j}[1+b_{j}]$ & $v_{j}[1+b_{j}]$\\
\end{tabular}
\end{center}
Hence by Example \ref{4.4.} $P(\Sigma_{\mathcal{H}_i})\cup v_{j}[1]$ spreads to $v_j[*]$ and $P(\Sigma_{\mathcal{H}_i})\cup v_{j}[1]$ is a plague on $\Sigma_{\mathcal{H}_{j}}$. This implies that $p_{j}N=p_iN+1$. In this case we have $imm(\Sigma_{\mathcal{H}_{j}})\leq  1/4$ since $n_{j}=n_i+3$ and

\begin{center}
$\frac{p_{{j}}}{n_{{j}}}=\frac{p_{i}N+1}{(n_{i}+3)N}\leq \frac{ \frac{n_{i}}{4}N+1}{(n_{i}+3)N}=\frac{n_{i}N+4}{4(n_{i}+3)N}\leq  1/4$.
\end{center}
Finally suppose that $m_{j}=2$ and $m=4$. Then we have $p_{j}=p_i+1$ and $n_{j}=n_i+4$. In this case we again have $imm(\Sigma_{\mathcal{H}_{j}})\leq  1/4$ since $\frac{p_{{j}}}{n_{{j}}}=\frac{p_{i}+1}{n_{i}+4}\leq \frac{ \frac{n_{i}}{4}+1}{(n_{i}+4)N}= 1/4$.

\end{proof}

\begin{center}
\section{\textsc{Proof of Theorem 3.13.}} {\label{5}}
\end{center}

\paragraph{} This section contains a case-by-case analysis of the coverings, immunities
and weights for pointed Schreier graphs $\mathcal{G}$ of finite homogeneous $\mathbf{PSL}(2,\mathbb{Z})$-spaces $\overline{\Sigma}$ such that $V_x\neq \emptyset$ and $V_{xy}=\emptyset$. The main goal of
this section is to prove Theorem \ref{3.13.} with following two main cases and their subcases.

\begin{description}
  \item[${}$ ${}$ ${}$ ${}$ ${}$ Case 1.] $\mathcal{G}$ with $V_x\neq \emptyset$ and $V_y= \emptyset=V_{xy}$.
  \item[${}$ ${}$ ${}$ ${}$ ${}$ Case 2.] $\mathcal{G}$ with $V_x\neq \emptyset \neq V_y$ and $V_{xy}=\emptyset$.
\end{description}

\subsection{\textbf{Case 1}. Pointed Schreier Graphs $\mathcal{G}$ with $V_x\neq \emptyset$ and $V_y=\emptyset=V_{xy}$.} {\label{Case 1.}}
\paragraph{} Suppose that $\Sigma$ is a covering with simply intersecting cycles of a finite homogeneous $\mathbf{PSL}(2,\mathbb{Z})$-space $\overline{\Sigma}$. Let $\mathcal{G}$ be the pointed Schreier graph of $\overline{\Sigma}$ with $t$ triangles and let $V_x \neq \emptyset $ and $V_y= \emptyset=V_{xy}$. Let $v_0\in V_x$ is the distinguished vertex of $\mathcal{G}$. Let $\mathcal{H}_0\prec_{m_1} \mathcal{H}_1\prec_{m_2}...\prec_{m_t} \mathcal{H}_t$ is the sequence of robust subgraphs of $\mathcal{G}$ with $v_0\in \mathcal{H}_j$ for all $j$ with $0\leq j\leq t$. Note that for $t=0$, $\mathcal{G}$ has no covering $\Sigma$ with simply intersecting cycles (see section 7.2 of \cite{22}). Since $\mathcal{G}$ is with $V_x\neq \emptyset$ and $V_y=\emptyset=V_{xy}$, it is easy to see that there is one $\mathcal{G}$ with $(t, n)=(1, 6)$ and $|V_x|=3$. However, by using Corollary \ref{2.10.}, such a $\mathcal{G}$ has no covering $\Sigma$ with simply intersecting cycles. By inspection one can see that the smallest given $\mathcal{G}$ which has a covering $\Sigma$ with simply intersecting cycles is for $(t, n)=(2, 8)$. For the covering $\Sigma$ of this $\mathcal{G}$, $imm(\Sigma)\leq 1/4$ (see section 7.12 of \cite{22}). Motivated by the example of $\mathcal{G}$ with $(t, n)=(2, 8)$, we prove the following result.

\begin{lem} {\label{5.1.}}
Let $\Sigma$ be a covering with simply intersecting cycles of a finite homogeneous $\mathbf{PSL}(2,\mathbb{Z})$-space $\overline{\Sigma}$ of size $n$. Let $\mathcal{G}$ be the pointed Schreier graph of $\overline{\Sigma}$ with $t$ triangles and $V_x\neq \emptyset$ and $V_y= \emptyset=V_{xy}$. Assume that $\mathcal{G}$ contains at least one robust subgraph $\mathcal{H}_{i}$ for some $2\leq i\leq t$ such that $\mathcal{H}_{i}\prec_{1}\mathcal{H}_{i+1}$. Then $imm(\Sigma)\leq 1/4$.
\end{lem}

\begin{proof}
Let $\mathcal{H}_0\prec_{m_1} \mathcal{H}_1\prec_{m_2}...\prec_{m_t} \mathcal{H}_t$ is a sequence of robust subgraphs of $\mathcal{G}$ with $v_0\in \mathcal{H}_j$ for all $j$ with $1\leq j\leq t$. Since $V_y= \emptyset=V_{xy}$, $i \neq 0,1$. By Lemma \ref{4.1.} $v_0[*] \cup P_K$ is a plague on $\Sigma_{\mathcal{H}_j}$, where $P_K=\{v(k)[*]|k\in K \}$. Assume that $|\Sigma_{\mathcal{H}_j}|=n_jN$ and $|P(\Sigma_{\mathcal{H}_j})|=p_jN$ for all $j$, where $N$ is the size of fiber over any point of $\mathcal{H}_{j}$. Suppose that there exists no $0\leq i^\prime \leq i-1$ such that $\mathcal{H}_{i^\prime }\prec_{1}\mathcal{H}_{i^\prime +1}$. Then by induction on $i^\prime$ it follows that
\begin{center}
$imm(\Sigma_{\mathcal{H}_{i^\prime}})\leq \frac{n_{i^\prime}+2}{4n_{i^\prime}}$,
\end{center}
where $n_{i^\prime}=|\mathcal{H}_{i^\prime}|$. The base step of the induction is as follows. For $i^\prime=0$, we have $\mathcal{H}_0$ with $n_0=2$. Since $V_y= \emptyset$, $n_0\neq 1$. Since $P=v_0[*]$ is a plague on $\mathcal{H}_0$, $imm(\Sigma_{\mathcal{H}_0})\leq 1/2  =\frac{n_0+2}{4n_0}$. For $i^\prime=1$, we have $\mathcal{H}_1$ with $n_1=6$. Note that $n_1\neq 4,5$ since $V_y= \emptyset=V_{xy}$. Now by Lemma \ref{4.1.} $v_0[*]\cup v(1)[*]$ is a plague on $\Sigma_{\mathcal{H}_1}$ and hence $imm(\Sigma_{\mathcal{H}_1})\leq 2N/6N = \frac{n_1+2}{4n_1}$. Note that $imm(\Sigma_{\mathcal{H}_{i^\prime}})\leq \frac{n_{i^\prime}+2}{4n_{i^\prime}}$ implies that $p_{i^\prime}=\frac{1}{4}(n_{i^\prime}+2)$ for all $0\leq i^\prime \leq i-1$.
\\[1pt] ${}$ \; Now we show that $imm(\Sigma_{\mathcal{H}_j})\leq 1/4$ for $i\leq j\leq t$. For $j=i$ we have $\mathcal{H}_{i}\prec_{1}\mathcal{H}_{i+1}$ and $n_j=n_{i-1}+2$, $p_jN=p_{i-1}N$. For $j=i$ we have $imm(\Sigma_{\mathcal{H}_j})\leq 1/4$ since
\begin{center}
$\frac{p_j}{n_j}=\frac{p_{i-1}}{n_{i-1}+2} = \frac{\frac{1}{4}(n_{i-1}+2)}{n_{i-1}+2}= 1/4$.
\end{center}
Now by Lemma \ref{4.2.} we have $imm(\Sigma_{\mathcal{H}_j})\leq 1/4$ for $i< j\leq t$. Hence $imm(\Sigma)=imm(\Sigma_{\mathcal{H}_t})\leq 1/4$.
\end{proof}

\paragraph{} Now we consider the coverings $\Sigma$ of those pointed Schreier graphs $\mathcal{G}$ (with $V_x\neq \emptyset$ and $V_y= \emptyset=V_{xy}$) which have no robust subgraph $\mathcal{H}_i$ such that $\mathcal{H}_i\prec_{1}\mathcal{H}_{i+1}$. It is easy to see that the smallest such pointed Schreier graph $\mathcal{G}$ is with $3$ triangles and $\mathcal{H}_0\prec_{2} \mathcal{H}_1\prec_{2}\mathcal{H}_2\prec_{0} \mathcal{H}_3$. There are two such pointed Schreier graphs which we discuss in the following subsections.

\subsubsection{\textbf{The graph {$\mathcal{G}_{10\{10\}}$}.}}
\begin{lem} {\label{5.2.}}
 Let $\Sigma$ be a covering of $\mathcal{G}_{10\{10\}}$ in Figure \ref{Figure 7.} with simply intersecting cycles. Then $imm(\Sigma)\leq 1/4=\omega(\Sigma).$
\end{lem}
\begin{proof}
To prove that $\omega(\Sigma) = 1/4$ observe that in every covering all cycles have length $9$.
From the following table it follows that $P=v_1[*]\cup v_3[*]\cup v_{9}[0]$ is a plague. \begin{center}
  \begin{tabular}{c|cccccccccc}
  pivot & $v_1[*]$ & $v_2[*]$ & $v_3[*]$& $v_4[*]$& $v_6[0]$ & $v_5[0]$ & $v_{9}[0]$ & $v_{10}[0]$ &...\\
  \hline
   & $v_2[*]$ & $v_4[*]$ & $v_{10}[*]$& $v_5[*]$& $v_7[1]$ & $v_6[1]$ & $v_8[1]$ & $v_{9}[1]$ &...\\
\end{tabular}
\end{center}
Thus $imm(\Sigma)< (2N+1)/10N<1/4=\omega(\Sigma)$ for all odd $N$ with $N>1$.
\end{proof}

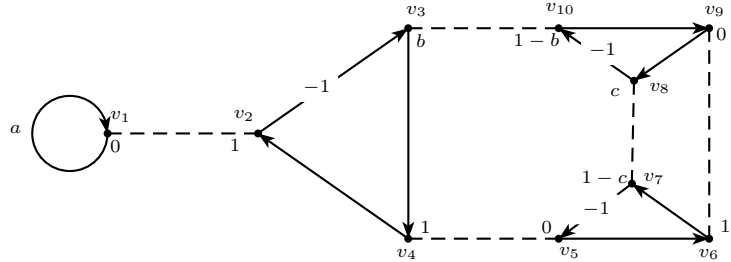
\begin{figure}[hb]
\centering

\psset{xunit=1.cm,yunit=0.7cm,algebraic=true,dimen=middle,dotstyle=o,dotsize=3pt 0,linewidth=0.8pt,arrowsize=3pt 2,arrowinset=0.25}
\begin{pspicture*}(-3.789322406870999,-2.7797908753611025)(6.5875102630716285,2.7214947863421752)
\psline[linestyle=dashed](-2.,0.)(0.,0.)
\psline{->}(0.,0.)(2.,2.)
\psline{->}(2.,2.)(2.,-2.)
\psline{->}(2.,-2.)(0.,0.)
\psline{<-}(4.,2.)(5.,1.)
\psline{<-}(5.,1.)(6.,2.)
\psline{<-}(6.,2.)(4.,2.)
\psline{->}(4.971018741623279,-0.9547198027581194)(4.,-2.)
\psline{->}(4.,-2.)(6.,-2.)
\psline{->}(6.,-2.)(4.971018741623279,-0.9547198027581194)
\psline[linestyle=dashed](2.,2.)(4.,2.)
\psline[linestyle=dashed](2.,-2.)(4.,-2.)
\psline[linestyle=dashed](5.,1.)(4.971018741623279,-0.9547198027581194)
\psline[linestyle=dashed](6.,2.)(6.,-2.)
\begin{scriptsize}
\psdots[dotstyle=*](0.,0.)
\rput[bl](-0.3216873689253473,0.24461261638098378){$v_2$}
\psdots[dotstyle=*](2.,2.)
\rput[bl](1.972687693775535,2.226118352349937){$v_3$}
\psdots[dotstyle=*](2.,-2.)
\rput[bl](1.8423254743038942,-2.3887042169461776){$v_4$}
\rput[bl](-0.37383225671400366,-0.3289811492942395){$1$}
\psdots[dotstyle=*](-2.,0.)
\rput[bl](-1.9642513342680246,0.21854017248665544){$v_1$}
\rput[bl](-1.9642513342680246,-0.35505359318856783){$0$}
\rput[bl](0.5,0.596423713739192){\psframebox*{$-1$}}
\rput[bl](2.1551948010358326,-1.8933277829539394){$1$}
\psdots[dotstyle=*](4.,2.)
\rput[bl](3.823831210272838,2.2782632401385934){$v_{10}$}
\psdots[dotstyle=*](5.,1.)
\rput[bl](5.205670736672233,0.7660614942675504){$v_8$}
\psdots[dotstyle=*](6.,2.)
\rput[bl](5.935699165713423,2.226118352349937){$v_{9}$}
\psdots[dotstyle=*](4.971018741623279,-0.9547198027581194)
\rput[bl](5.127453404989248,-0.9547198027581194){$v_7$}
\psdots[dotstyle=*](6.,-2.)
\rput[bl](5.857481834030438,-2.3887042169461776){$v_6$}
\psdots[dotstyle=*](4.,-2.)
\rput[bl](4.006338317533135,-2.362631773051849){$v_5$}
\rput[bl](3.4,1.6264521427803853){$1-b$}
\rput[bl](6.092133829079392,1.7828868061463552){$0$}
\rput[bl](4.3,1.3){\psframebox*{$-1$}}
\rput[bl](4.219207644265074,-1.7283135684333427){\psframebox*{$-1$}}
\rput[bl](4.3,-0.9807922466524478){$1-c$}
\rput[bl](6.144278716868048,-1.867255339059611){$1$}
\rput[bl](2.103049913247176,1.6525245866747136){$b$}
\rput[bl](3.7716863224841815,-1.8933277829539394){$0$}
\rput[bl](4.684221858785669,0.7139166064788938){$c$}
\pnode(-2,0){A}
\nccircle[angleA=90]{<-}{A}{.5cm}
\rput[bl](-3.3,0.){$a$}

\end{scriptsize}
\end{pspicture*}
\caption{Schreier graph $\mathcal{G}_{10\{10\}}$ and its coverings} \label{Figure 7.}
\end{figure}

\subsubsection{\textbf{The graph {$\mathcal{G}_{10\{5,\;3, \; 2\}}$}.}}
\begin{lem}{\label{5.3.}}
 Let $\Sigma$ be a covering of $\mathcal{G}_{10\{5,\;3, \; 2\}}$ in Figure \ref{Figure 8.} with simply intersecting cycles. Then $imm(\Sigma)\leq\omega(\Sigma).$
\end{lem}
\begin{proof}
 The $xy$-cycles with their labels are: $(v_1\;v_3 \; v_8 \; v_5\; v_2)$ with $a-b-c+2$, $(v_4\;v_6 \; v_{10})$ with $b$, $(v_7\; v_9)$ with $c$. The $yx$-cycles with their labels are: $(v_1\;v_2 \; v_{10} \; v_7\; v_4)$ with $a-b-c+2$, $(v_3\;v_5 \; v_9)$ with $b$, $(v_6\; v_8)$ with $c$. By Lemma \ref{2.7.} $N>1$. The cycle structure on each vertex of $\Sigma$ is the following: $v_1[i], v_2[i]$ have two $5$-cycles; $v_3[i], v_4[i], v_5[i], v_{10}[i]$ have cycles of length $5$ and $3\left|\left<b\right>\right|$; $ v_6[i],v_9[i]$ have cycles of length $2\left|\left<c\right>\right|$ and $3\left|\left<b\right>\right|$;
$v_7[i], v_8[i]$ have cycles of length $5$ and $2\left|\left<c\right>\right|$. Using this cycle structure the weight of $\Sigma$ is the following:
\[ \omega(\Sigma) = \left\{
  \begin{array}{l l}
    3/10 & \text{if $b \equiv c\equiv 0 \pmod N$},
    \\ 17/60 & \text{if $\left|\left <b\right>\right|\equiv0 \pmod N$ and $c\not\equiv 0 \pmod N$},
    \\11/40 & \text{if $\left|\left <b\right>\right| \not\equiv0 \pmod N$  and $c\equiv 0 \pmod N$},
    \\1/4 & \text{otherwise}.
  \end{array} \right.\]
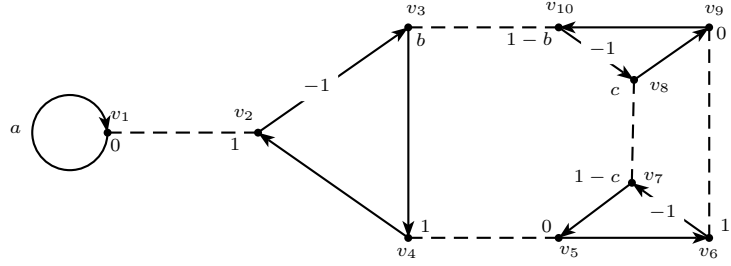
\begin{figure}[hb]
\centering
\psset{xunit=1.cm,yunit=0.7cm,algebraic=true,dimen=middle,dotstyle=o,dotsize=3pt 0,linewidth=0.8pt,arrowsize=3pt 2,arrowinset=0.25}
\begin{pspicture*}(-3.789322406870999,-2.7797908753611025)(6.5875102630716285,2.7214947863421752)
\psline[linestyle=dashed](-2.,0.)(0.,0.)
\psline{->}(0.,0.)(2.,2.)
\psline{->}(2.,2.)(2.,-2.)
\psline{->}(2.,-2.)(0.,0.)
\psline{->}(4.,2.)(5.,1.)
\psline{->}(5.,1.)(6.,2.)
\psline{->}(6.,2.)(4.,2.)
\psline{->}(4.971018741623279,-0.9547198027581194)(4.,-2.)
\psline{->}(4.,-2.)(6.,-2.)
\psline{->}(6.,-2.)(4.971018741623279,-0.9547198027581194)
\psline[linestyle=dashed](2.,2.)(4.,2.)
\psline[linestyle=dashed](2.,-2.)(4.,-2.)
\psline[linestyle=dashed](5.,1.)(4.971018741623279,-0.9547198027581194)
\psline[linestyle=dashed](6.,2.)(6.,-2.)
\begin{scriptsize}
\psdots[dotstyle=*](0.,0.)
\rput[bl](-0.3216873689253473,0.24461261638098378){$v_2$}
\psdots[dotstyle=*](2.,2.)
\rput[bl](1.972687693775535,2.226118352349937){$v_3$}
\psdots[dotstyle=*](2.,-2.)
\rput[bl](1.8423254743038942,-2.3887042169461776){$v_4$}
\rput[bl](-0.37383225671400366,-0.3289811492942395){$1$}
\psdots[dotstyle=*](-2.,0.)
\rput[bl](-1.9642513342680246,0.21854017248665544){$v_1$}
\rput[bl](-1.9642513342680246,-0.35505359318856783){$0$}
\rput[bl](0.5,0.596423713739192){\psframebox*{$-1$}}
\rput[bl](2.1551948010358326,-1.8933277829539394){$1$}
\psdots[dotstyle=*](4.,2.)
\rput[bl](3.823831210272838,2.2782632401385934){$v_{10}$}
\psdots[dotstyle=*](5.,1.)
\rput[bl](5.205670736672233,0.7660614942675504){$v_8$}
\psdots[dotstyle=*](6.,2.)
\rput[bl](5.935699165713423,2.226118352349937){$v_9$}
\psdots[dotstyle=*](4.971018741623279,-0.9547198027581194)
\rput[bl](5.127453404989248,-0.9547198027581194){$v_7$}
\psdots[dotstyle=*](6.,-2.)
\rput[bl](5.857481834030438,-2.3887042169461776){$v_6$}
\psdots[dotstyle=*](4.,-2.)
\rput[bl](4.006338317533135,-2.362631773051849){$v_5$}
\rput[bl](3.3,1.6264521427803853){$1-b$}
\rput[bl](6.092133829079392,1.7828868061463552){$0$}
\rput[bl](4.3,1.3){\psframebox*{$-1$}}
\rput[bl](5.1,-1.796283135684333427){\psframebox*{$-1$}}
\rput[bl](4.199581209858501507,-0.9807922466524478){$1-c$}
\rput[bl](6.144278716868048,-1.867255339059611){$1$}
\rput[bl](2.103049913247176,1.6525245866747136){$b$}
\rput[bl](3.7716863224841815,-1.8933277829539394){$0$}
\rput[bl](4.684221858785669,0.7139166064788938){$c$}
\pnode(-2,0){A}
\nccircle[angleA=90]{<-}{A}{.5cm}
\rput[bl](-3.3,0.){$a$}

\end{scriptsize}
\end{pspicture*}
\caption{Schreier graph $\mathcal{G}_{10\{5,\;3, \; 2\}}$ and its coverings} \label{Figure 8.}
\end{figure}
\paragraph{}
Assume first that $b \equiv c\equiv 0 \pmod N$. Now by the following table $P=v_1[*]\cup v_3[*]\cup v_6[*]$ is a plague.

\begin{center}
  \begin{tabular}{c|cccccccccc}
  pivot & $v_1[*]$ & $v_2[*]$ & $v_4[*]$& $v_5[*]$& $v_7[*]$ & $v_{10}[*]$ & $v_3[*]$\\
  \hline
   & $v_2[*]$ & $v_4[*]$ & $v_5[*]$& $v_7[*]$& $v_8[*]$ & $v_9[*]$ & $v_{10}[*]$\\
\end{tabular}
\end{center}
In this case $imm(\Sigma)\leq3/10=\omega(\Sigma).$ Next assume that $b\not\equiv 0 \pmod N$, and $I$ is a set of representatives of $\mathbb{Z}_N/\left <b\right>$. Then by the following table $P=v_1[*]\cup v_8[*]\cup v_3[I]$, and by Example \ref{4.4.}, it follows that $P$ spreads to $v_3[*]$.
\begin{center}
  \begin{tabular}{c|ccccccc}
  pivot & $v_1[*]$ & $v_4[I]$ & $v_7[I-1]$ & $v_9[I-1]$ & $v_3[I-1+c]$ & $v_2[I+c]$ \\
  \hline
   & $v_2[*]$ & $v_{5}[I]$ & $v_6[I]$ & $v_{10}[I]$ & $v_4[I+c]$ & $v_3[I+c]$\\
\end{tabular}
\end{center}
In this case $imm(\Sigma)\leq (2N+\left|I|\right )/10N\leq(2N+N/2)/10N=1/4\leq\omega(\Sigma).$ Next assume that $\left|\left <c\right>\right| \not\equiv0 \pmod N$, and $I$ is a set of representatives of $\mathbb{Z}_N/\left <c\right>$. Then we claim that $P=v_1[*]\cup v_3[*]\cup v_8[I]$ is plague. We compute
\begin{center}
  \begin{tabular}{c|cc}
  pivot &  $v_9[I]$ & $v_7[I]$\\
  \hline
   &  $v_6[I+1]$ & $v_8[I+c]$\\
\end{tabular}
\end{center}
and hence by Example \ref{4.4.} $P$ spreads to $v_8[*]$. Thus $P=v_1[*]\cup v_3[*]\cup v_8[I]$ is plague. In this case $imm(\Sigma)\leq (2N+\left|I|\right )/10N\leq(2N+N/2)/10N=1/4\leq\omega(\Sigma).$

\end{proof}

\subsubsection{\textbf{Pointed Schreier Graphs with $V_x\neq \emptyset$, $V_y=\emptyset=V_{xy}$ and the Fragments.}}
\paragraph{} In this section we discuss those pointed Schreier graphs $\mathcal{G}$ with $V_x\neq \emptyset$ and $V_y=\emptyset=V_{xy}$ which contain the fragments of the robust subgraph $\mathcal{H}_0$ of the graphs $\mathcal{G}_{10\{10\}}$ and $\mathcal{G}_{10\{5,\;3, \; 2\}}$. We write these fragments as $\mathcal{F}_1:=\mathcal{G}_{10\{10\}}\setminus \mathcal{H}_0(\mathcal{G}_{10\{10\}})$ and $\mathcal{F}_2:=\mathcal{G}_{10\{5,\;3, \; 2\}}\setminus \mathcal{H}_0(\mathcal{G}_{10\{5,\;3, \; 2\}})$,
as shown in Figure \ref{Figure 9.}. Note that $|\mathcal{F}_1|=|\mathcal{F}_2|=8$ and $\mathcal{F}_1$ (resp. $\mathcal{F}_2$) has a free site for one vertex which can be used for gluing $\mathcal{F}_1$ with other subgraphs.


\begin{figure}[hb]
\centering
\psset{xunit=.45cm,yunit=0.45cm,algebraic=true,dimen=middle,dotstyle=o,dotsize=3pt 0,linewidth=0.8pt,arrowsize=3pt 2,arrowinset=0.25}
\begin{pspicture*}(-0.48,-2.58)(14.38,2.42)
\psline{->}(0,0)(2,2)
\psline{->}(2,2)(2,-2)
\psline{->}(2,-2)(0,0)
\psline{->}(4.97,-0.95)(4,-2)
\psline{->}(4,-2)(6,-2)
\psline{->}(6,-2)(4.97,-0.95)
\psline[linestyle=dashed](2,2)(4,2)
\psline[linestyle=dashed](2,-2)(4,-2)
\psline[linestyle=dashed](5,1)(4.97,-0.95)
\psline[linestyle=dashed](6,2)(6,-2)
\psline{->}(4,2)(6,2)
\psline{->}(6,2)(5,1)
\psline{->}(5,1)(4,2)
\psline{->}(10,2)(10,-2)
\psline{->}(13,-1)(12,-2)
\psline{->}(12,-2)(14,-2)
\psline{->}(14,-2)(13,-1)
\psline[linestyle=dashed](10,2)(12,2)
\psline[linestyle=dashed](10,-2)(12,-2)
\psline[linestyle=dashed](13,1)(13,-1)
\psline[linestyle=dashed](14,2)(14,-2)
\psline{->}(8,0)(10,2)
\psline{->}(10,-2)(8,0)
\psline{->}(12,2)(13,1)
\psline{->}(13,1)(14,2)
\psline{->}(14,2)(12,2)
\begin{scriptsize}
\psdots[dotsize=5pt 0](0,0)
\psdots[dotstyle=*](2,2)
\psdots[dotstyle=*](2,-2)
\psdots[dotstyle=*](4,2)
\psdots[dotstyle=*](5,1)
\psdots[dotstyle=*](6,2)
\psdots[dotstyle=*](4.97,-0.95)
\psdots[dotstyle=*](6,-2)
\psdots[dotstyle=*](4,-2)
\rput[bl](3.58,-2.6){$\mathcal{F}_1$}
\psdots[dotstyle=*](10,2)
\psdots[dotstyle=*](10,-2)
\psdots[dotstyle=*](12,2)
\psdots[dotstyle=*](13,1)
\psdots[dotstyle=*](14,2)
\psdots[dotstyle=*](13,-1)
\psdots[dotstyle=*](14,-2)
\psdots[dotstyle=*](12,-2)
\rput[bl](11.59,-2.6){$\mathcal{F}_2$}
\psdots[dotsize=5pt 0](8,0)
\end{scriptsize}
\end{pspicture*}
\caption{Fragments $\mathcal{F}_1$ and $\mathcal{F}_2$} \label{Figure 9.}
\end{figure}
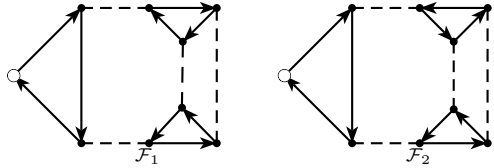
\paragraph{} Let $k$ be an integer with $0\leq k < t$ and $\mathcal{H}_0\prec_2 \mathcal{H}_1\prec_2...\prec_2 \mathcal{H}_{k}$ is a sequence robust subgraphs of a pointed Schreier graph $\mathcal{G}$ with $t$ triangles. Then $\mathcal{H}_{k}$ has $k+1$ possible open $y-$edges which have a vertex without $x-$edge or $x-$loop on it. For example the robust subgraph $\mathcal{H}_{7}$ is shown in Figure \ref{Figure 10.}.
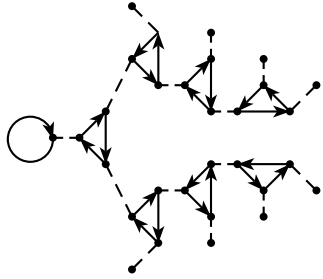
\begin{figure}[hb!]
\centering
\psset{xunit=.7cm,yunit=.7cm,algebraic=true,dimen=middle,dotstyle=o,dotsize=3pt 0,linewidth=0.8pt,arrowsize=3pt 2,arrowinset=0.25}
\begin{pspicture*}(-3.01,-2.73)(3.4,2.79)
\psline[linestyle=dashed](-1,-0.5)(-0.5,-1.5)
\psline{->}(-1.5,0)(-1,0.5)
\psline{->}(-1,0.5)(-1,-0.5)
\psline{->}(-1,-0.5)(-1.5,0)
\psline[linestyle=dashed](-2,0)(-1.5,0)
\psline[linestyle=dashed](-1,0.5)(-0.5,1.5)
\psline{->}(-0.5,-1.5)(0,-1)
\psline{->}(0,-1)(0,-2)
\psline{->}(0,-2)(-0.5,-1.5)
\psline{->}(-0.5,1.5)(0,1)
\psline{->}(0,1)(0,2)
\psline{->}(0,2)(-0.5,1.5)
\psline[linestyle=dashed](0,-1)(0.5,-1)
\psline[linestyle=dashed](0,1)(0.5,1)
\psline[linestyle=dashed](0,2)(-0.5,2.5)
\psline[linestyle=dashed](0,-2)(-0.5,-2.5)
\psline{->}(0.5,1)(1,1.5)
\psline{->}(1,1.5)(1,0.5)
\psline{->}(1,0.5)(0.5,1)
\psline{->}(0.5,-1)(1,-1.5)
\psline{->}(1,-1.5)(1,-0.5)
\psline{->}(1,-0.5)(0.5,-1)
\psline[linestyle=dashed](1,1.5)(1,2)
\psline[linestyle=dashed](1,0.5)(1.5,0.5)
\psline[linestyle=dashed](1,-0.5)(1.5,-0.5)
\psline[linestyle=dashed](1,-1.5)(1,-2)
\psline{->}(1.5,0.5)(2.5,0.5)
\psline{->}(2.5,0.5)(2,1)
\psline{->}(2,1)(1.5,0.5)
\psline{->}(1.5,-0.5)(2,-1)
\psline{->}(2,-1)(2.5,-0.5)
\psline{->}(2.5,-0.5)(1.5,-0.5)
\psline[linestyle=dashed](2.5,0.5)(3,1)
\psline[linestyle=dashed](2,1)(2,1.5)
\psline[linestyle=dashed](2,-1)(2,-1.5)
\psline[linestyle=dashed](2.5,-0.5)(3,-1)
\begin{scriptsize}
\psdots[dotstyle=*](-1.5,0)
\psdots[dotstyle=*](-1,0.5)
\psdots[dotstyle=*](-1,-0.5)
\psdots[dotstyle=*](-0.5,-1.5)
\psdots[dotstyle=*](-2,0)
\psdots[dotstyle=*](-0.5,1.5)
\psdots[dotstyle=*](0,-1)
\psdots[dotstyle=*](0,-2)
\psdots[dotstyle=*](0,1)
\psdots[dotstyle=*](0.5,-1)
\psdots[dotstyle=*](-0.5,-2.5)
\psdots[dotstyle=*](0.5,1)
\psdots[dotstyle=*](-0.5,2.5)
\psdots[dotstyle=*](1,1.5)
\psdots[dotstyle=*](1,0.5)
\psdots[dotstyle=*](1,-1.5)
\psdots[dotstyle=*](1,-0.5)
\psdots[dotstyle=*](1.5,-0.5)
\psdots[dotstyle=*](1,-2)
\psdots[dotstyle=*](1,2)
\psdots[dotstyle=*](1.5,0.5)
\psdots[dotstyle=*](2.5,0.5)
\psdots[dotstyle=*](2,1)
\psdots[dotstyle=*](2.5,-0.5)
\psdots[dotstyle=*](2,-1)
\psdots[dotstyle=*](3,1)
\psdots[dotstyle=*](3,-1)
\psdots[dotstyle=*](2,-1.5)
\psdots[dotstyle=*](2,1.5)
\end{scriptsize}
\end{pspicture*}
\pnode(-5.57,2.7){A}
\nccircle[angleA=90]{<-}{A}{.3cm}
\caption{The Robust Subgraph $\mathcal{H}_{7}$} \label{Figure 10.}
\end{figure}
\paragraph{} Note that any pointed Schreier graph $\mathcal{G}$ with $V_x\neq \emptyset$, $V_y= \emptyset=V_{xy}$ which has no robust subgraph $\mathcal{H}_i$ such that $\mathcal{H}_i\prec_{1}\mathcal{H}_{i+1}$ can be generated by gluing $k+1$ copies of $\mathcal{F}_1$ (resp. $\mathcal{F}_2$ or both $\mathcal{F}_1$ and $\mathcal{F}_2$) with $k+1$ possible open $y-$edges of $\mathcal{H}_{k}$. Therefore we write any such graph as $\mathcal{G}=Span(\mathcal{H}_{k},\mathcal{F}_1, \mathcal{F}_2)$. For example $\mathcal{G}_{10\{10\}}=Span(\mathcal{H}_{0},\mathcal{F}_1)$ and $\mathcal{G}_{10\{5,\;3, \; 2\}}=Span(\mathcal{H}_{0},\mathcal{F}_2)$. Note that there are infinitely many pointed Schreier graphs $\mathcal{G}=Span(\mathcal{H}_{k},\mathcal{F}_1, \mathcal{F}_2)$ since there are infinitely many subgraphs $\mathcal{H}_{k}$ with $\mathcal{H}_0\prec_2 \mathcal{H}_1\prec_2...\prec_2 \mathcal{H}_{k}$. Note also that if $\Sigma$ is a covering of $\mathcal{G}=Span(\mathcal{H}_{k},\mathcal{F}_1, \mathcal{F}_2)$ with simply intersecting cycles then $\mathcal{G}=Span(\mathcal{H}_{k},\mathcal{F}_1, \mathcal{F}_2)$ is with only one $x-$loop, that is, $|V_x|=1$. If $|V_x|>1$ then $\mathcal{H}_{k}$ has at least two fixed points of $x$ which will also be on certain $xy-$ and $yx-$ cycles. Therefore by Corollary \ref{2.10.}, $\mathcal{G}=Span(\mathcal{H}_{k},\mathcal{F}_1, \mathcal{F}_2)$ with $|V_x|>1$ has no covering with simply intersecting cycles.
\paragraph{} In the following lemma we estimate immunities on the coverings of $\mathcal{G}=Span(\mathcal{H}_{k},\mathcal{F}_1, \mathcal{F}_2)$ by the same method which we used to calculate immunities on the coverings of $\mathcal{G}_{10\{10\}}=Span(\mathcal{H}_{0},\mathcal{F}_1)$ and $\mathcal{G}_{10\{5,\;3, \; 2\}}=Span(\mathcal{H}_{0},\mathcal{F}_2)$.

\begin{lem}{\label{5.4.}}
Let $\Sigma$ be a covering of $\mathcal{G}=Span(\mathcal{H}_{k},\mathcal{F}_1, \mathcal{F}_2)$ with simply intersecting cycles. Then $imm(\Sigma)\leq \omega(\Sigma).$
\end{lem}
\begin{proof}
First suppose that $\mathcal{G}=Span(\mathcal{H}_{k},\mathcal{F}_1, \mathcal{F}_2)$ contains at least one copy of the fragment $\mathcal{F}_1$. Since $|V_x|=1$ and $t=4k+3$, $|\mathcal{G}|=3t+1=12k+10$. Note that $\mathcal{H}_k\prec_2 \mathcal{H}_{k+1}\prec_2 \mathcal{H}_{k+2}\prec_0 \mathcal{H}_{k+3}$. Therefore $k+1, k+2\in K$, where $K=\{k\in J:\mathcal{H}_{k-1}\prec_2 \mathcal{H}_k\}$ with $J$ as a set of subscripts of all robust subgraphs $\mathcal{H}_j$ of $\mathcal{G}$. Also $v(k+1)[*], v(k+2)[*]\in P_K$ where $P_K=\{v(k)[*]|k\in K \}\subset \Sigma$ and $v:K\rightarrow V(\mathcal{G})$ be a map such that $v(k)\in \mathcal{H}_k\backslash \mathcal{H}_{k-1}$ for any $k\in K$. Assume that the label of $y-$edge at $v(k+2)$ is $0$, and $P_{K\setminus \{k+2\}}=P_K\setminus v(k+2)[*]$. Then the set $v_0[*] \cup P_{K\setminus \{k+2\}}\cup v(k+2)[0]$ is a plague on the covering $\Sigma$ of $\mathcal{G}$. Therefore for any covering $\Sigma$ of $\mathcal{G}$ with at least one copy of the fragment $\mathcal{F}_1$, we have
\begin{center}
$imm(\Sigma)\leq \frac{(3k+2)N+1}{(12k+10)N} \leq 1/4\leq \omega(\Sigma).$
\end{center}
\paragraph{} Now suppose that $\mathcal{G}=Span(\mathcal{H}_{k},\mathcal{F}_1, \mathcal{F}_2)$ has no copy of the fragment $\mathcal{F}_1$, that is, $\mathcal{G}=Span(\mathcal{H}_{k},\mathcal{F}_2)$. Since $\mathcal{G}=Span(\mathcal{H}_{k},\mathcal{F}_2)$ contains $k+1$ copies of $\mathcal{F}_2$, any $xy$-cycle (resp. $yx$-cycle) has length $\geq 2$. Therefore $(\omega^\prime_{2j})_{j>1}=1/3$ (resp. $(\omega^\prime_{i2})_{i>1}=1/3$). Note that the number of vertices on cycles of length two are $2k+2$ and the number of vertices on cycles of length $\geq 3$ are $n-2k-2$ for $0\leq k < t$. Therefore we have

\begin{center}
  $\omega(\Sigma) \geq \frac{1}{n}[(2k+2)(1/3)+(n-2k-2)(7/24)]=\frac{k}{12n}+\frac{1}{12n}+\frac{7}{24}$,
\end{center}
where $n=|\mathcal{G}|=12k+10$. Let $v(k+1), v(k+2)\in P_K$ such that the labels of $y-$edges at $v(k+1)$ and $v(k+2)$ are $b$ and $c$ respectively. If $b=0=c$, the set $v_0[*] \cup P_{K}$ is a plague on the covering $\Sigma$ of $\mathcal{G}=Span(\mathcal{H}_{k},\mathcal{F}_2)$. Therefore for $b=0=c$ we have
\begin{center}
$imm(\Sigma)\leq \frac{(3k+2)N+N}{(12k+10)N} \leq \omega(\Sigma).$
\end{center}
For $b\neq 0$, $v_0[*] \cup P_{K\setminus \{k+1\}}\cup v(k+1)[I]$ with $|I|\leq N/2$ is a plague on the covering $\Sigma$ of $\mathcal{G}=Span(\mathcal{H}_{k},\mathcal{F}_2)$. Therefore for $b\neq0$ we have
\begin{center}
$imm(\Sigma)\leq \frac{(3k+2)N+N/2}{(12k+10)N} = 1/4\leq \omega(\Sigma).$
\end{center}
Similarly, for $c\neq 0$, $v_0[*] \cup P_{K\setminus \{k+2\}}\cup v(k+2)[I]$ with $|I|\leq N/2$ is a plague on the covering $\Sigma$ of $\mathcal{G}=Span(\mathcal{H}_{k},\mathcal{F}_2)$. Therefore for $c\neq0$ we have
\begin{center}
$imm(\Sigma)\leq \frac{(3k+2)N+N/2}{(12k+10)N} = 1/4\leq \omega(\Sigma).$
\end{center}
\end{proof}

\subsection{\textbf{Case 2.} Pointed Schreier Graphs $\mathcal{G}$ with $V_x \neq \emptyset \neq V_y$ and $V_{xy}= \emptyset$.} \label{Case 2.}

\paragraph{} In this section we study the coverings $\Sigma$ with simply intersecting cycles of finite homogeneous $\mathbf{PSL}(2,\mathbb{Z})$-space $\overline{\Sigma}$ whose pointed Schreier graphs $\mathcal{G}$ are with $V_x \neq \emptyset \neq V_y$ and $V_{xy}= \emptyset$. Note that for $t=0$, $\mathcal{G}$ with $V_x \neq \emptyset \neq V_y$ and $V_{xy}= \emptyset$ has only one point, and in this case $imm(\Sigma)=\omega(\Sigma)$ (see section 7.1 of \cite{22}). For $t=1$ there are two cases of $\mathcal{G}$ with $V_x \neq \emptyset \neq V_y$ and $V_{xy}= \emptyset$, namely (1) $\mathcal{G}$ with four points including one point with $x$-loop and two points with $y$-loop (see section 7.6 of \cite{22}), and (2) $\mathcal{G}$ with five points including two points with $x$-loop and one point with $y$-loop. However, by using Corollary \ref{2.10.}, such a $\mathcal{G}$ has no covering $\Sigma$ with simply intersecting cycles. By inspection one can see that the smallest non-trivial $\mathcal{G}$ with $V_x \neq \emptyset \neq V_y$ and $V_{xy}= \emptyset$ which has a covering $\Sigma$ with simply intersecting cycles is for $(t, n)=(2, 7)$ with $|V_x|=|V_y|=1$. In this case we have two possible graphs, namely, $\mathcal{G}_{7\{4,\;3\}}$ and $\mathcal{G}_{7\{5,\;2\}}$. For graph $\mathcal{G}_{7\{4,\;3\}}$ we have $imm(\Sigma)\leq\omega(\Sigma)$ (see section 7.11 of \cite{22}). We discuss the graph $\mathcal{G}_{7\{5,\;2\}}$ here.

\subsubsection{\textbf{The graph $\mathcal{G}_{7\{5,\;2\}}$.}}
\begin{lem}
Let $\Sigma$ be a covering of $\mathcal{G}_{7\{5,\;2\}}$ in Figure \ref{Figure 11.} with simply intersecting cycles. Then $N = 7$ and $(a, b, c) = (2, 4, 1)$.
\end{lem}
\begin{proof}
The $xy$-cycles with their labels are: $(v_1\;v_3 \; v_7 \; v_5 \; v_2)$ with $a + b + c$, $(v_4 \; v_6)$ with $1-c$. The $yx$-cycles with their labels are: $(v_1\;v_2 \; v_6 \; v_7 \; v_4)$ with $a + b + c$, $(v_3 \; v_5)$ with $1-c$. By Lemma \ref{2.6.} on $v_1$, $v_7$ we have $3a \equiv -1 \pmod N$ and $2b \equiv 1 \pmod N$. This implies that $N$ is odd and is not a multiple of $3$. Lemma \ref{2.8.} on $v_1$ implies $a+b+c=0$. Since $\Sigma$ is a covering with simply intersecting cycles, it follows that: $|\{v_4[k(1-c)], v_6[k(1-c)]\} \cap \{{v_4[0],v_6[a+c]}\}| \leq1.$ This implies $a+c=-b \notin \{k(1-c)|k\in Z \}.$ From this the claim follows.
\begin{figure}[hb]
\begin{center}
\psset{xunit=1.0cm,yunit=1.0cm,algebraic=true,dimen=middle,dotstyle=o,dotsize=3pt 0,linewidth=0.8pt,arrowsize=3pt 2,arrowinset=0.25}
\begin{pspicture*}(-3.5,-1.54)(5.32,1.52)
\psline{->}(3.,-1.)(3.,1.)
\psline{->}(3.,1.)(4.,0.)
\psline{->}(4.,0.)(3.,-1.)
\psline[linestyle=dashed](1.,-1.)(3.,-1.)
\psline[linestyle=dashed](1.,1.)(3.,1.)
\psline[linestyle=dashed](-2.,0.)(0.,0.)
\psline{->}(0.,0.)(1.,1.)
\psline{->}(1.,1.)(1.,-1.)
\psline{->}(1.,-1.)(0.,0.)
\pscircle[linestyle=dashed](4.44,-0.02){0.4404543109109052}
\begin{scriptsize}
\psdots[dotstyle=*](0.,0.)
\rput[bl](-0.24,0.18){$v_2$}
\psdots[dotstyle=*](1.,1.)
\rput[bl](0.86,1.22){$v_3$}
\psdots[dotstyle=*](1.,-1.)
\rput[bl](0.88,-1.3){$v_4$}
\psdots[dotstyle=*](3.,1.)
\rput[bl](2.84,1.22){$v_6$}
\psdots[dotstyle=*](3.,-1.)
\rput[bl](2.9,-1.32){$v_5$}
\psdots[dotstyle=*](4.,0.)
\rput[bl](3.82,0.3){$v_7$}
\rput[bl](2.7,-0.08){\psframebox*{$-1$}}
\rput[bl](2.66,-0.9){$1$}
\rput[bl](1.1,-0.86){$0$}
\rput[bl](2.74,0.74){$c$}
\psdots[dotstyle=*](-2.,0.)
\rput[bl](-1.94,0.22){$v_1$}
\rput[bl](-1.98,-0.28){$0$}
\rput[bl](0.1,0.18){\psframebox*{$-1$}}
\rput[bl](1.06,0.74){$1-c$}
\rput[bl](-0.22,-0.26){$1$}
\rput[bl](5.04,-0.06){$b$}
\pnode(-2,0){A}
\nccircle[angleA=90]{<-}{A}{.5cm}
\rput[bl](-3.25,0.){$a$}
\end{scriptsize}
\end{pspicture*}
\end{center}
 \caption{Schreier graph $\mathcal{G}_{7\{5,\;2\}}$ and it coverings} \label{Figure 11.}
\end{figure}
\end{proof}

\begin{lem} {\label{5.6.}}
Let $\Sigma$ be a covering of $\mathcal{G}_{7\{5,\;2\}}$ with simply intersecting cycles. Then $imm(\Sigma)\leq\omega(\Sigma).$
\end{lem}

\begin{proof}
The cycle structure on each vertex of $\Sigma$ is the following:
\begin{align*}
  v_1[i], v_2[i],v_7[i]: \;& \text{two 5-cycles,}
\\ v_3[i],v_4[i],v_5[i], v_6[i]: \;& \text{cycles of length 5 and 2,}
\end{align*}
for all $i\in \mathbb{Z}_7.$ From the following table it follows that $P=v_1[*]\cup v_3[*]$ is a plague.
\begin{center}
  \begin{tabular}{c|ccccc}
  pivot & $v_1[*]$ & $v_2[*]$ & $v_3[*]$& $v_4[*]$& $v_6[*]$\\
  \hline
   & $v_2[*]$ & $v_4[*]$ & $v_6[*]$& $v_5[*]$& $v_7[*]$\\
\end{tabular}
\end{center}
Therefore $imm(\Sigma)\leq2/7<25/84=\omega(\Sigma).$

\end{proof}

\subsubsection{\textbf{Pointed Schreier Graphs with $V_x\neq \emptyset \neq V_y$, $V_{xy}=\emptyset$ and the fragments.}}
\paragraph{} In this section we discuss those pointed Schreier graphs $\mathcal{G}$ with $V_x\neq \emptyset \neq V_y$, $V_{xy}=\emptyset$ which contain the fragments of robust subgraph $\mathcal{H}_0$ of the graphs $\mathcal{G}_{10\{10\}}$, $\mathcal{G}_{10\{5,\;3, \; 2\}}$, $\mathcal{G}_{7\{5,\;2\}}$, and $\mathcal{G}_{7\{4,\;3\}}$. We write these fragments as $\mathcal{F}_1:=\mathcal{G}_{10\{10\}}\setminus \mathcal{H}_0(\mathcal{G}_{10\{10\}})$ and $\mathcal{F}_2:=\mathcal{G}_{10\{10\}}\setminus \mathcal{H}_0(\mathcal{G}_{10\{5,\;3, \; 2\}})$, $\mathcal{F}_3:=\mathcal{G}_{7\{5,\;2\}}\setminus \mathcal{H}_0(\mathcal{G}_{7\{5,\;2\}})$, and $\mathcal{F}_4:=\mathcal{G}_{7\{4,\;3\}}\setminus \mathcal{H}_0(\mathcal{G}_{7\{4,\;3\}})$ as shown in Figure \ref{Figure 9.} and Figure \ref{Figure 12.}. Note that $|\mathcal{F}_1|=|\mathcal{F}_2|=8$, $|\mathcal{F}_3|=|\mathcal{F}_4|=5$.

\newpage

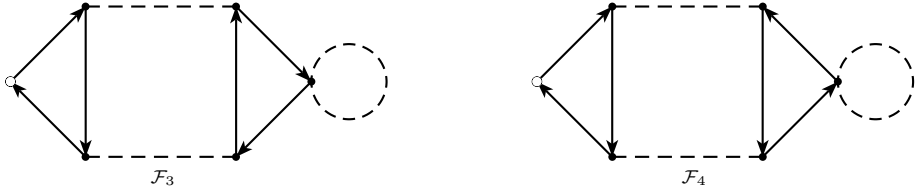
\begin{figure}[hb]
\centering
\psset{xunit=1.0cm,yunit=1.0cm,algebraic=true,dimen=middle,dotstyle=o,dotsize=3pt 0,linewidth=0.8pt,arrowsize=3pt 2,arrowinset=0.25}
\begin{pspicture*}(-2.3,-1.58)(10.2,1.3)
\psline{->}(-2,0)(-1,1)
\psline{->}(-1,1)(-1,-1)
\psline{->}(-1,-1)(-2,0)
\psline[linestyle=dashed](-1,1)(1,1)
\psline[linestyle=dashed](-1,-1)(1,-1)
\pscircle[linestyle=dashed](2.5,0){0.5}
\psline{->}(1,1)(2,0)
\psline{->}(2,0)(1,-1)
\psline{->}(1,-1)(1,1)
\psline{->}(5,0)(6,1)
\psline{->}(6,1)(6,-1)
\psline{->}(6,-1)(5,0)
\psline[linestyle=dashed](6,1)(8,1)
\psline[linestyle=dashed](6,-1)(8,-1)
\pscircle[linestyle=dashed](9.5,0){0.5}
\psline{->}(8,1)(8,-1)
\psline{->}(8,-1)(9,0)
\psline{->}(9,0)(8,1)
\begin{scriptsize}
\psdots[dotsize=4pt 0](-2,0)
\psdots[dotstyle=*](1,-1)
\psdots[dotstyle=*](-1,-1)
\psdots[dotstyle=*](-1,1)
\psdots[dotstyle=*](1,1)
\rput[bl](-0.14,-1.38){$\mathcal{F}_3$}
\psdots[dotstyle=*](2,0)
\psdots[dotsize=4pt 0](5,0)
\psdots[dotstyle=*](8,-1)
\psdots[dotstyle=*](6,-1)
\psdots[dotstyle=*](6,1)
\psdots[dotstyle=*](8,1)
\rput[bl](6.92,-1.38){$\mathcal{F}_4$}
\psdots[dotstyle=*](9,0)
\end{scriptsize}
\end{pspicture*}

\caption{Fragments $\mathcal{F}_3$ and $\mathcal{F}_4$} \label{Figure 12.}
\end{figure}


\paragraph{} Let $k$ be an integer with $0\leq k < t$ and $\mathcal{H}_0\prec_2 \mathcal{H}_1\prec_2...\prec_2 \mathcal{H}_{k}$ is a sequence of a pointed Schreier graph $\mathcal{G}$. Then $\mathcal{H}_{k}$ has $k+1$ possible open $y-$edges (which have a vertex without $x-$edge). Note that any pointed Schreier graphs $\mathcal{G}$ with $V_x\neq \emptyset \neq V_y$ and $V_{xy}=\emptyset$, and without a robust subgraph $\mathcal{H}_i$ such that $\mathcal{H}_i\prec_{1}\mathcal{H}_{i+1}$, can be generated by gluing $k$ copies of $\mathcal{F}_1$ (resp. $\mathcal{F}_2$ or both $\mathcal{F}_1$ and $\mathcal{F}_2$) and one copy of $\mathcal{F}_1$ (resp. $\mathcal{F}_4$) with $k+1$ possible open $y-$edges of $\mathcal{H}_{k}$. Therefore we write any such graph as $\mathcal{G}=Span(\mathcal{H}_{k},\mathcal{F}_1, \mathcal{F}_2, \mathcal{F}_3)$ or $\mathcal{G}=Span(\mathcal{H}_{k},\mathcal{F}_1, \mathcal{F}_2, \mathcal{F}_4)$.

\paragraph{} In the following lemma we estimate immunities on the coverings of $\mathcal{G}=Span(\mathcal{H}_{k},\mathcal{F}_1, \mathcal{F}_2, \mathcal{F}_3)$ (resp. $\mathcal{G}=Span(\mathcal{H}_{k},\mathcal{F}_1, \mathcal{F}_2, \mathcal{F}_4)$) by the same method which we used to calculate immunities on the coverings of $\mathcal{G}=Span(\mathcal{H}_{k},\mathcal{F}_1, \mathcal{F}_2)$ in Lemma \ref{5.4.}.

\begin{lem}{\label{5.7.}}
 Let $\Sigma$ be a covering of $\mathcal{G}=Span(\mathcal{H}_{k},\mathcal{F}_1, \mathcal{F}_2, \mathcal{F}_3)$ with simply intersecting cycles. Then $imm(\Sigma)\leq \omega(\Sigma).$
\end{lem}
\begin{proof}
First suppose that $\mathcal{G}=Span(\mathcal{H}_{k},\mathcal{F}_1, \mathcal{F}_2, \mathcal{F}_3)$ contains at least one copy of the fragment $\mathcal{F}_1$. Since $|V_x|=1$, $|\mathcal{G}|=12k+7$. Note that $\mathcal{H}_k\prec_2 \mathcal{H}_{k+1}\prec_2 \mathcal{H}_{k+2}\prec_0 \mathcal{H}_{k+3}$. Therefore $k+1, k+2\in K$ and $v(k+1)[*], v(k+2)[*]\in P_K$. Assume that the label of $y-$edge at $v(k+2)$ is $0$, and $P_{K\setminus\{k+2\}}=P_K\setminus v(k+2)[*]$. Then the set $v_0[*] \cup P_{K\setminus\{k+2\}}\cup v(k+2)[0]$ is a plague on the covering $\Sigma$ of $\mathcal{G}$. Therefore for any covering $\Sigma$ of $\mathcal{G}$ with at least one copy of the fragment $\mathcal{F}_1$, we have
\begin{center}
$imm(\Sigma)\leq \frac{(3k+1)N+1}{(12k+7)N} \leq 1/4\leq \omega(\Sigma).$
\end{center}
Now suppose that $\mathcal{G}=Span(\mathcal{H}_{k},\mathcal{F}_1, \mathcal{F}_2, \mathcal{F}_3)$ has no copy of the fragment $\mathcal{F}_1$ and $\mathcal{G}=Span(\mathcal{H}_{k}, \mathcal{F}_2, \mathcal{F}_3)$. Since $\mathcal{G}=Span(\mathcal{H}_{k}, \mathcal{F}_2, \mathcal{F}_3)$ contains $k$ copies of $\mathcal{F}_2$ and one copy of $\mathcal{F}_3$, any $xy$-cycle (resp. $yx$-cycle) has length $\geq 2$ which implies that $(\omega^\prime_{2j})_{j>1}=1/3$ (resp. $(\omega^\prime_{i2})_{i>1}=1/3$). Note that for $\mathcal{G}=Span(\mathcal{H}_{k}, \mathcal{F}_2, \mathcal{F}_3)$ the number of vertices on cycles of length two are $2k+2$ and the number of vertices on cycles of length $\geq 3$ are $n-2k-2$ for $0< k < t$. Therefore for $\mathcal{G}=Span(\mathcal{H}_{k}, \mathcal{F}_2, \mathcal{F}_3)$ we have

\begin{center}
  $\omega(\Sigma) \geq \frac{1}{n}[(2k+2)(1/3)+(n-2k-2)(7/24)]=\frac{k}{12n}+\frac{1}{12n}+\frac{7}{24}$,
\end{center}
where $n=|\mathcal{G}|=12k+10$. Also for $\mathcal{G}=Span(\mathcal{H}_{k}, \mathcal{F}_2, \mathcal{F}_4)$ the number of vertices on cycles of length two are $2k$ and the number of vertices on cycles of length $\geq 3$ are $n-2k$ for $0\leq k < t$. Therefore for $\mathcal{G}=Span(\mathcal{H}_{k}, \mathcal{F}_2, \mathcal{F}_4)$ we have

\begin{center}
  $\omega(\Sigma) \geq \frac{1}{n}[(2k)(1/3)+(n-2k)(7/24)]=\frac{k}{12n}+\frac{7}{24}$,
\end{center}
Let $v(k+1), v(k+2)\in P_K$ such that the labels of $y-$edges at $v(k+1)$ and $v(k+2)$ are $b$ and $c$ respectively. If $b=0=c$, then the set $v_0[*] \cup P_{K}$ is a plague on the covering $\Sigma$ of $\mathcal{G}=Span(\mathcal{H}_{k},\mathcal{F}_2)$. Therefore for $b=0=c$ we have
\begin{center}
$imm(\Sigma)\leq \frac{(3k+1)N+N}{(12k+7)N} \leq \omega(\Sigma).$
\end{center}
For $b\neq 0$, $v_0[*] \cup P_{K\setminus\{k+1\}}\cup v(k+1)[I]$ with $|I|\leq N/2$ is a plague on the covering $\Sigma$ of $\mathcal{G}=Span(\mathcal{H}_{k},\mathcal{F}_2)$. Therefore for $b\neq0$ we have
\begin{center}
$imm(\Sigma)\leq \frac{(3k+1)N+N/2}{(12k+7)N} = 1/4\leq \omega(\Sigma).$
\end{center}
Similarly, for $c\neq 0$, $v_0[*] \cup P_{K\setminus\{k+2\}}\cup v(k+2)[I]$ with $|I|\leq N/2$ is a plague on the covering $\Sigma$ of $\mathcal{G}=Span(\mathcal{H}_{k},\mathcal{F}_2)$. Therefore for $b\neq0$ we have
\begin{center}
$imm(\Sigma)\leq \frac{(3k+1)N+N/2}{(12k+7)N} = 1/4\leq \omega(\Sigma).$
\end{center}
\end{proof}

\paragraph{} Now we discuss the pointed Schreier graphs $\mathcal{G}$ of $\overline{\Sigma}$ with $V_x \neq \emptyset \neq V_y$ and $V_{xy}= \emptyset$ for which $imm(\Sigma)\leq 1/4$. We begin with the following remark.

\begin{rem}\label{6.8.} Let $\Sigma$ be a covering with simply intersecting cycles of finite homogeneous $\mathbf{PSL}(2,\mathbb{Z})$-space $\overline{\Sigma}$. Let $\mathcal{G}$ be the pointed Schreier graph of $\overline{\Sigma}$ with $V_x \neq \emptyset \neq V_y$ and $V_{xy}= \emptyset$. In this case we again consider $v_0\in V_x$ as a distinguished vertex. We write $V_y=\{v_{i_1}, v_{i_2},...,v_{i_k}\}$, where $v_{i_k}\in \mathcal{H}_{i_k}\setminus \mathcal{H}_{i_k-1}$ for positive integers $i_1, i_2,..., i_k$ with $i_1<i_2<...<i_k$. Note that $\mathcal{H}_{i_k-1}\prec_{m_{i_k}} \mathcal{H}_{i_k}$ with $m_{i_k}\in \{1,2\}$. Moreover we have the following observations.
\begin{description}
  \item[(1).] It is easy to see that for pointed Schreier graphs $\mathcal{G}$ with $V_x \neq \emptyset \neq V_y$, $V_{xy}= \emptyset$ and $|V(\mathcal{G})|\neq 1$, $|C_{xy}(v_0)\cap C_{yx}(v_0)|>1$, where $C_{xy}(v_0)$ and $C_{yx}(v_0)$ are the $xy-$ and $yx-$cycles of $\mathcal{G}$ containing $v_0\in V_x$. Therefore, by Lemma \ref{2.7.}, every covering $\Sigma$ with simply intersecting cycles of $\mathcal{G}$ is non-trivial, that is, $N>1$, where $N$ is the size of fiber over any point of $\mathcal{G}$.
  \item[(2).] If $a_0$ is the label of $x-$loop at $v_0\in V_x$ and and $b_k$ is the label of $y-$loop at $v_{i_k}\in V_y$ then by Lemma \ref{2.6.}, $3a_0 \equiv -1 \pmod N$, $2b_k \equiv 1 \pmod N$, and $b_k\neq 0,1$, since $N>1$. This implies that $N$ is odd and is not a multiple of $3$.
  \item[(3).] If $V_x\neq \emptyset \neq V_y$ and $V_{xy}= \emptyset$, then $V_y\cap \mathcal{H}_{i_1-1}=\emptyset$. Therefore, by Lemma \ref{5.1.}, we have
\[ imm(\mathcal{H}_{i_1-1}) \leq \left\{
  \begin{array}{l l}
    \frac{n_{i_1-1}+2}{4n_{i_1-1}}, & \text{if there is no $\mathcal{H}_{i}$ for $2\leq i\leq i_1-1$ such that $\mathcal{H}_{i}\prec_{1}\mathcal{H}_{i+1}$},\\
   1/4, & \text{if there is a $\mathcal{H}_{i}$ for $2\leq i\leq i_1-1$ such that $\mathcal{H}_{i}\prec_{1}\mathcal{H}_{i+1}$}.
  \end{array} \right.\]
\end{description}

\end{rem}

\begin{lem}{\label{5.9.}}
Let $\Sigma$ be a covering with simply intersecting cycles of finite homogeneous $\mathbf{PSL}(2,\mathbb{Z})$-space $\overline{\Sigma}$. Let $\mathcal{G}$ be the pointed Schreier graph of $\overline{\Sigma}$ such that $V_x\neq \emptyset$, $V_{xy}= \emptyset$ and $V_y=\{v_{i_1}\}$, where $v_{i_1}\in \mathcal{H}_{i_1}\setminus \mathcal{H}_{i_1-1}$, and $\mathcal{H}_{i_1-1}\prec_{1}\mathcal{H}_{i_1}$. Assume further that there exist a robust subgraph $\mathcal{H}_{i}$ for $i\neq i_1$ such that $\mathcal{H}_{i-1}\prec_{1}\mathcal{H}_{i}$. Then $imm(\Sigma)\leq 1/4$.
\end{lem}

\begin{proof}
Let $t$ be the number of triangles of $\mathcal{G}$ and $\mathcal{H}_0\prec_{m_1} \mathcal{H}_1\prec_{m_2}...\prec_{m_t} \mathcal{H}_t$ is a finite sequence of robust subgraphs $\mathcal{H}_j$ of $\mathcal{G}$. Now choose $i_1=t$. Then $V_y\cap \mathcal{H}_{t-1}= \emptyset$ and $1\leq i\leq t-1$. Since there exist a robust subgraph $\mathcal{H}_{i}$ for $i\neq i_1$ such that $\mathcal{H}_{i-1}\prec_{1}\mathcal{H}_{i}$, therefore by Lemma \ref{5.1.}, we have $imm(\Sigma_{\mathcal{H}_{t-1}}) \leq 1/4.$ This implies that there exists a plague $P(\Sigma_{\mathcal{H}_{t-1}})$ consisting of complete fibers over $p_{t-1}$ points of $\mathcal{H}_{t-1}$ such that $p_{t-1}\leq \frac{n_{t-1}}{4}$. Since $\mathcal{H}_{i_1-1}\prec_{1}\mathcal{H}_{i_1}$, we have $n_{t}=n_{t-1}+1$ and $p_{t}=p_{t-1}$. Therefore
\begin{center}
  $\frac{p_{t}}{n_{t}}=\frac{p_{t-1}}{n_{t-1}+1}< \frac{p_{t-1}}{n_{t-1}} \leq 1/4$.
\end{center}
Therefore $imm(\Sigma)=imm(\Sigma_{\mathcal{H}_{t}}) \leq 1/4$.
\end{proof}

\begin{lem}{\label{5.10.}}
Let $\Sigma$ be a covering with simply intersecting cycles of finite homogeneous $\mathbf{PSL}(2,\mathbb{Z})$-space $\overline{\Sigma}$. Let $\mathcal{G}$ be the pointed Schreier graph of $\overline{\Sigma}$ such that $V_x\neq \emptyset$, $V_{xy}= \emptyset$ and $V_y=\{v_{i_1}\}$, where $v_{i_1}\in \mathcal{H}_{i_1}\setminus \mathcal{H}_{i_1-1}$ and $\mathcal{H}_{i_1-1}\prec_{2}\mathcal{H}_{i_1}$. Then $imm(\Sigma)\leq 1/4$.
\end{lem}

\begin{proof}
Let $t$ be the number of triangles of $\mathcal{G}$ and $\mathcal{H}_0\prec_{m_1} \mathcal{H}_1\prec_{m_2}...\prec_{m_t} \mathcal{H}_t$ is a finite sequence of robust subgraphs $\mathcal{H}_j$ of $\mathcal{G}$. By Remark \ref{6.8.} (3) we have:
\begin{center}
  $\frac{p_{i_1-1}}{n_{i_1-1}}\leq \frac{n_{i_1-1}+2}{4n_{i_1-1}}.$
\end{center}
This implies that there exists a plague $P(\Sigma_{\mathcal{H}_{i_1-1}})$ consisting of complete fibers over $p_{i_1-1}$ points of $\mathcal{H}_{i_1-1}$ such that $p_{i_1-1}\leq \frac{n_{i_1-1}}{4}$. Therefore $p_{i_1-1}\leq \frac{n_{i_1-1}+2}{4}$. Since $\mathcal{H}_{i_1-1}\prec_{2} \mathcal{H}_{i_1}$, $n_{i_1}=n_{i_1-1}+3$. Now consider the Figure \ref{Figure 13.}.

\begin{figure}[hb]
\centering

\psset{xunit=.8cm,yunit=.8cm,algebraic=true,dimen=middle,dotstyle=o,dotsize=3pt 0,linewidth=0.8pt,arrowsize=3pt 2,arrowinset=0.25}
\begin{pspicture*}(-1.74,-2.22)(14.08,2.54)
\psline[linestyle=dashed](0,0)(2,0)
\psline(2,1)(4,1)
\psline(4,1)(4,-1)
\psline(4,-1)(2,-1)
\psline(2,-1)(2,1)
\psline{->}(6,0)(7,1)
\psline{->}(7,1)(8,0)
\psline{->}(8,0)(6,0)
\psline[linestyle=dashed](4,0)(6,0)
\psline[linestyle=dashed](8,0)(10,0)
\pscircle[linestyle=dashed](7,1.55){0.45}
\psline(10,1)(10,-1)
\psline(10,1)(12,1)
\psline(12,1)(12,-1)
\psline(12,-1)(10,-1)
\begin{scriptsize}
\psdots[dotstyle=*](0,0)
\rput[bl](0.14,-0.28){$v_0$}
\psdots[dotstyle=*](2,0)
\psdots[dotstyle=*](4,0)
\rput[bl](4.06,-0.45){$v^\prime_{i_1-1}$}
\psdots[dotstyle=*](6,0)
\rput[bl](5.94,-0.39){$v_{i_1-1}$}
\psdots[dotstyle=*](7,1)
\rput[bl](6.25,0.2) {\psframebox*{$-1$}}
\rput[bl](6.8,1.1){$v_{i_1}$}
\psdots[dotstyle=*](8,0)
\rput[bl](8.06,-0.45){$v^\prime_{i_1}$}
\psdots[dotstyle=*](10,0)
\rput[bl](6.94,2.2){$b_{i_1}$}

\pnode(0,0){A}
\nccircle[angleA=90]{<-}{A}{.5cm}
\end{scriptsize}
\end{pspicture*}
\caption{Robust Subgraph $\mathcal{H}_{i_1}$} \label{Figure 13.}
\end{figure}
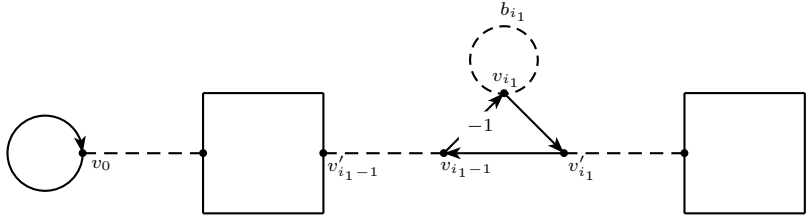
By the method of Lemma \ref{4.2.}, $P(\mathcal{H}_{i_1})=P(\mathcal{H}_{i_1-1})\cup v_{i_1}[1]$ spreads to $v_{i_1}[*]$, and therefore $P(\mathcal{H}_{i_1})=P(\mathcal{H}_{i_1-1})\cup v_{i_1}[1]$ is a plague on $\Sigma_{\mathcal{H}_{i_1}}$. Therefore $p_{i_1}N=p_{i_1-1}N+1\leq \frac{(n_{i_1-1}+2)N+4}{4}$, and for $N\geq 5$ we have
\begin{center}
  $\frac{p_{i_1}N}{n_{i_1}N}=\frac{p_{i_1-1}N+1}{(n_{i_1}+3)N} \leq \frac{(n_{i_1-1}+2)N+4}{4(n_{i_1-1}+3)N}< 1/4$.
\end{center}
Therefore $imm(\Sigma_{\mathcal{H}_{i_1}})< 1/4$. Now by Lemma \ref{4.2.}, $imm(\Sigma_{\mathcal{H}_{j}})\leq 1/4$ for all $i_1<j\leq t$ and hence $imm(\Sigma)\leq 1/4$.

\end{proof}

\begin{lem}{\label{5.11.}}
Let $\Sigma$ be a covering with simply intersecting cycles of finite homogeneous $\mathbf{PSL}(2,\mathbb{Z})$-space $\overline{\Sigma}$. Let $\mathcal{G}$ be the pointed Schreier graph of $\overline{\Sigma}$ such that $V_x\neq \emptyset$, $V_{xy}= \emptyset$, $|V_y|\geq2$. Then $imm(\Sigma)\leq 1/4$.
\end{lem}

\begin{proof}
Let $t$ be the number of triangles of $\mathcal{G}$ and $\mathcal{H}_0\prec_{m_1} \mathcal{H}_1\prec_{m_2}...\prec_{m_t} \mathcal{H}_t$ is a finite sequence of robust subgraphs $\mathcal{H}_j$ of $\mathcal{G}$. Let $V_y=\{v_{i_1}, v_{i_2},...,v_{i_k}\}$, where $v_{i_k}\in \mathcal{H}_{i_k}\setminus \mathcal{H}_{i_k-1}$ for positive integers $i_1, i_2,..., i_k$.  Now we have two cases to consider. First suppose that there exists at least one $i_k \in J$ such that $\mathcal{H}_{i_k-1}\prec_{2} \mathcal{H}_{i_k}$. In this case we can choose such $i_k$ as $i_1$. Then by using the arguments of Lemma \ref{5.10.} it follows that $imm(\Sigma_{\mathcal{H}_{i_1}})\leq 1/4$. Now by Lemma \ref{4.2.} $imm(\Sigma_{\mathcal{H}_{j}})\leq 1/4$ for all $i_1<j\leq t$. Therefore for $m_{i_1}=2$ we have $imm(\Sigma)\leq 1/4$.

\paragraph{} Next suppose that there does not exist any $i_k \in J$ such that $\mathcal{H}_{i_k-1}\prec_{2} \mathcal{H}_{i_k}$. In this case we can choose the subscripts $i_1, i_2,..., i_k$ such that $i_k=i_{k-1}+1$ for all $k$. In particular consider $i_2=i_1+1$. Now by Remark \ref{6.8.} (3) we have $imm(\Sigma_{\mathcal{H}_{i_1-1}}) \leq \frac{n_{i_1-1}+2}{4n_{i_1-1}}.$ This implies that there exists a plague $P(\Sigma_{\mathcal{H}_{i_1-1}})$ consisting of complete fibers over $p_{i_1-1}$ points of $\mathcal{H}_{i_1-1}$ such that $p_{i_1-1}\leq \frac{n_{i_1-1}+2}{4}$. Therefore $p_{i_1}=p_{i_1-1}\leq \frac{n_{i_1-1}+2}{4}$. Now since $i_2=i_1+1$, $\mathcal{H}_{i_1}\prec_{1} \mathcal{H}_{i_2}$ and therefore $n_{i_2}=n_{i_1}+1=n_{i_1-1}+2$. Also $P(\Sigma_{\mathcal{H}_{i_2}})=P(\Sigma_{\mathcal{H}_{i_1}})=P(\Sigma_{\mathcal{H}_{i_1-1}})$. This implies that $p_{i_2}=p_{i_1}=p_{i_1-1}\leq \frac{n_{i_1-1}+2}{4}$. Therefore
\begin{center}
$\frac{p_{i_2}}{n_{i_2}}= \frac{p_{i_1-1}}{n_{i_1-1}+2}\leq \frac{n_{i_1-1}+2}{4(n_{i_1-1}+2)}=1/4$.
\end{center}
This implies that $imm(\Sigma_{\mathcal{H}_{i_2}}) \leq 1/4$. Therefore by Lemma \ref{4.2.} $imm(\Sigma_{\mathcal{H}_{j}})\leq 1/4$ for all $i_1<j\leq t$ and hence $imm(\Sigma)\leq 1/4$.

\end{proof}

\paragraph{\textbf{The Proof of Theorem \ref{3.13.}}.} The pointed Schreier graphs $\mathcal{G}$ to consider are with $V_x\neq \emptyset$ and $V_{xy}=\emptyset$. There are two main cases to consider, namely the Case 1 in Section \ref{Case 1.} when $\mathcal{G}$ is with $V_x\neq \emptyset$ and $V_y= \emptyset=V_{xy}$, and Case 2 in Section \ref{Case 2.} when $\mathcal{G}$ is with $\mathcal{G}$ with $V_x\neq \emptyset \neq V_y$ and $V_{xy}=\emptyset$. In Case 1 there are two types of $\mathcal{G}$, namely, $\mathcal{G}$ with at least one robust subgraph $\mathcal{H}_{i}$ for some $2\leq i\leq t$ such that $\mathcal{H}_{i}\prec_{1}\mathcal{H}_{i+1}$, and $\mathcal{G}$ without a robust subgraph $\mathcal{H}_{i}$ for some $2\leq i\leq t$ such that $\mathcal{H}_{i}\prec_{1}\mathcal{H}_{i+1}$. In both these cases we proved that the immunity can be bounded above by the weight. The corresponding claims are Lemmas \ref{5.1.}, \ref{5.2.}, \ref{5.3.} and \ref{5.4.}. In Case 2 there are five subcases to consider. In each of these cases we proved that the immunity can be bounded above by the weight. The corresponding claims are Lemmas \ref{5.6.}, \ref{5.7.}, \ref{5.9.}, \ref{5.10.}, and \ref{5.11.}.


\paragraph{\textbf{Conclusion.}} Theorem \ref{3.13.} implies that the Conjecture \ref{1.1.} is true for any covering $\Sigma$ with simply intersecting cycles of finite homogeneous $\mathbf{PSL}(2,\mathbb{Z})$-space $\overline{\Sigma}$ whose pointed Schreier graph is with $V_x\neq \emptyset$, $V_y= \emptyset=V_{xy}$ and $V_x\neq \emptyset \neq V_y$, $V_{xy}=\emptyset$. However through case-by-analysis it is concluded that the Conjecture \ref{1.1.} is open for coverings $\Sigma$ of the pointed Schreier graph $\mathcal{G}$ in the following cases:
\begin{multicols}{3}
\setlength{\columnsep}{0.25cm} \noindent
$\mathcal{G}$ with $V_x\neq \emptyset \neq V_{xy}$ and $V_y=\emptyset$,\\$\mathcal{G}$ with $V_y\neq \emptyset \neq V_{xy}$ and $V_x=\emptyset$,
\columnbreak
 \\$\mathcal{G}$ with $V_x\neq \emptyset \neq V_{xy}$ and $V_y\neq\emptyset$, \\ $\mathcal{G}$ with $V_{xy}\neq \emptyset$ and $V_x=\emptyset=V_y$,
\columnbreak
 \\$\mathcal{G}$ with $V_y\neq \emptyset$  and $V_x=\emptyset=V_{xy}$,\\$\mathcal{G}$ with $V_x= \emptyset = V_{y}$ and $V_{xy}=\emptyset$.
 \end{multicols}

\paragraph{\textbf{Acknowledgement.}} The author is grateful to Istvan Heckenberger for introducing the problem and useful discussion. This work is supported by German Academic Exchange Service (DAAD).


\begin{center}

\end{center}
Naqeeb ur Rehman, Allama Iqbal Open University Islamabad, Pakistan.\\
E-mail: naqeeb@aiou.edu.pk

\addcontentsline{toc}{section}{References}
\begin{thebibliography}{99}
\setlength{\itemsep}{0.85mm}






\bibitem{3} {Andruskiewitsch, N., Gra\~{n}a, M.: From racks to pointed Hopf algebras., Adv. Math., 178(2), 177–243 (2003).}

\bibitem{4} {Balogh, J., Bollob\'{a}s, B., Morris, R.: Graph bootstrap  percolation., Random Structures and Algorithms., 41(4), 413-440 (2012).}




\bibitem{8}{Brieskorn, E.: Automorphic sets and braids and singularities., In Braids (Santa Cruz, CA, 1986),
Contemp. Math., 45-115, Amer. Math. Soc., Providence, RI, (1988).}

\bibitem{9}{Ceccherini-Silberstein, T., Coornaert, M.: Cellular automata and groups., Springer Monographs in Mathematics., Springer-Verlag, Berlin, (2010)}.





\bibitem{15} {Gra\~{n}a, M., Heckenberger, I., Vendramin, L.: Nichols algebras of group type with many quadratic relations., Adv. Math., 227 (2011).}


\bibitem{21} {Heckenberger, I., Lochmann, A., Vendramin, L.: Braided racks, Hurwitz actions and Nichols algebras with many cubic relations., Transform. Groups., 17(1), 157–194 (2012).}
\bibitem{22} {Heckenberger, I., Lochmann, A., Vendramin, L.: Nichols algebras with many cubic relations., Trans. Am. Math. Soc., 367 (9), 6315–6356 (2015).}
\bibitem{23}{Hurwitz, A.: Ueber Riemann’sche Flächen mit gegebenen Verzweigungspunkten., Math. Ann.,
39 (1891).}



\bibitem{26}{Kassel, C., Turaev, V.: Braid groups., volume 247 of Graduate Texts in Mathematics., Springer, New York, (2008)}.



\bibitem{29}{Rankin, R. A.: Modular forms and functions., Cambridge University Press., Cambridge, (1977)}.





\end{thebibliography}
\end{document}